\documentclass[10pt]{article}
\author{G\'abor Pataki, Aleksandr Touzov \footnote{Department of Statistics and Operations Research, UNC Chapel Hill}} 
\title{An echelon form of weakly infeasible semidefinite programs and bad projections of the psd cone}

\usepackage[latin1]{inputenc}

\usepackage{graphicx}
\usepackage[margin=3.0cm]{geometry}

\usepackage{mathptmx}      
\usepackage{enumitem}
\usepackage[pagebackref=false, colorlinks=true, linkcolor=blue, citecolor=blue, linkcolor=blue, anchorcolor=red, urlcolor=blue]{hyperref}
\usepackage{amsmath,amsthm}
\usepackage{amsfonts}
\usepackage{float}
\usepackage{import}
\usepackage{tikz}
\usepackage{pgfplots}
\usepackage{algorithm}
\usepackage{algpseudocode}
\usepackage{cite}
\mathchardef\mhyphen="2D

\DeclareMathAlphabet{\mathcal}{OMS}{cmsy}{m}{n}
\newcommand{\co}[1]{}
\newcommand{\B}{\mathcal B}
\newcommand{\I}{\mathcal I}
\newcommand{\R}{\mathcal R}
\newcommand{\N}{\mathcal N}
\newcommand{\M}{\mathcal M}
\newcommand{\A}{\mathcal A}
\newcommand{\V}{\mathcal V}
\newcommand{\T}{\mathcal T}
\newcommand{\rad}[1]{\mathbb{R}^{#1}}
\newcommand{\sym}[1]{{\cal S}^{#1}}
\newcommand{\psd}[1]{{\cal S}_+^{#1}}
\newcommand{\psdn}{{\cal S}_+^{n}}

\newcommand{\eref}[1]{(\ref{#1})} 
\newcommand{\symn}{\sym{n}}

\newcommand{\np}{{\cal NP}}
\newcommand{\conp}{{\rm co \mhyphen}{\cal NP}}
\newcommand{\lin}{\operatorname{lin}}
\newcommand{\myrank}{\operatorname{rank}}
\makeatletter
\newcommand{\leqnomode}{\tagsleft@true}
\newcommand{\reqnomode}{\tagsleft@false}
\makeatother
\renewcommand*{\arraystretch}{1.0}
\newtheorem{example}{Example}
\newtheorem{lemma}{Lemma}
\newtheorem{definition}{Definition}
\newtheorem{remark}{Remark}
\newtheorem{theorem}{Theorem}
\newtheorem{corollary}{Corollary}
\newtheorem{case}{Case}

\begin{document}

\maketitle

\begin{abstract}
A weakly infeasible semidefinite program (SDP) has no feasible solution, but  it has approximate solutions whose constraint violation is arbitrarily small.
These SDPs are ill-posed and numerically  often unsolvable. They are also closely related to ``bad" linear projections that map the cone of positive semidefinite matrices to a nonclosed set. We describe   a simple echelon form of weakly infeasible SDPs with the following properties: (i) it is obtained by elementary row operations and congruence transformations, (ii) it makes weak infeasibility evident, and (iii) it permits us to construct any weakly infeasible SDP or bad linear projection by an elementary combinatorial algorithm. Based on our echelon form we generate a  library of  computationally very difficult  SDPs.
Finally, we show that some  SDPs in the literature are in our echelon form, for example, the SDP from the sum-of-squares relaxation of minimizing Motzkin's famous  polynomial.
\end{abstract}



\section{Introduction}
\label{intro}

Semidefinite programming (SDP) feasibility problems of the form 
\begin{equation}  \label{p} \tag{P}
\begin{array}{ccl}
\A X \!\!\!\!  &=& \!\!\!\!  b\\
X \!\!\!\!  & \in  & \!\!\!\!  \psd{n},  
\end{array} 
\end{equation} 
are fundamental in many areas of applied mathematics, including 
combinatorial optimization, polynomial optimization, control theory, and machine learning. Here $\A$ is a linear map from $n \times n$ symmetric matrices to $\rad{m}, \, b$ is in $\rad{m},  \, $ and 
$\psd{n} \, $ is the set of symmetric positive semidefinite (psd) matrices.

Semidefinite programs -- either in the feasibility form, or in an optimization form, with an objective function attached --  are often pathological  and  this work focuses on their 
pathological kind of infeasibility, called {\em weak infeasibility}.  
Precisely, we  say that \eqref{p} is weakly infeasible when it 
has no feasible solution, but the set of solutions of the linear system of equations  
$\A X = b\, $ has zero distance to $\psd{n}.$

\begin{example} \label{example-small}

An enlightening,  classical, and minimal example is 
\begin{equation}\label{problem-small} \tag{{ME}}
\begin{array}{rcl}
x_{11} &  = & 0 \\
x_{12}  = x_{21} & = & 1 \\
X  & \in  & \psd{2},
\end{array}
\end{equation}
where the $(i,j)$th element of $X$ is denoted by $x_{ij}.$ If $X$ satisfies the equality constraints of \eqref{problem-small}, then
$$
X =  \begin{pmatrix} 0 & 1 \\
	1            & x_{22}
	\end{pmatrix}
$$
for some $x_{22}$ real number, so $X$ cannot be positive semidefinite. Hence \eqref{problem-small} is infeasible.   
However, such $X$ matrices converge to $\psd{2}$   if we choose $x_{22} >0 \, $ to be large: then to make $X$ psd, we must only slightly change its $0$ entry to $1/x_{22}. \,$ 
So we conclude that  \eqref{problem-small}  is weakly infeasible. 
	
We visualize this example in Figure \ref{fig:asymptote}. The solid blue set bordered by a hyperbola is 
	$$
	S \, = \, \bigl\{ \, (x_{11}, x_{22}) \in \rad{2}_+~:~ x_{11} x_{22} \geq 1 \, \bigr\},
	$$
the set of diagonals of $2 \times 2$ psd matrices whose offdiagonal elements are $1.$ We approach $S$  arbitrarily closely  if we fix $x_{11}=0$ and make $x_{22}$ large, moving towards infinity on the $x_{22}$ axis.
\begin{figure}[H]
	\centering 
	\begin{tikzpicture}[scale=0.6, every node/.style={scale=0.7}]
	
    \fill [blue!10, domain=0.2:5, variable=\x]
        (0.2, 5)
        -- plot[samples = 100] ({\x}, {1/\x})
        -- (5, 0.2)
        -- (5,5)
        -- cycle;
    \draw [blue!40, thick]
        (0.2, 5)
        -- (5,5)
        -- (5, 0.2);
    \draw [very thick, blue!80, domain=0.2:5, variable=\x]
        plot[samples = 100] ({\x}, {1/\x});
        
    \fill (0,0) circle(2.75pt);
    
    \draw [->,draw=black,very thick] (-0.0,0) -- (5.1,0);
    \draw [->,draw=black,very thick] (0,-0.0) -- (0,5.1);
    
    

    \node[] at (5,-0.5) { \Large$x_{11}$};
    \node[] at (-0.6,5) { \Large$x_{22}$};
    \node[blue] at (2.75,2.75) { \huge$S$};
    
    \node[] at (-0.5,-0.5) { \large$(0,0)$};
	
	\end{tikzpicture}
	\caption{A visualization of \eqref{problem-small}} 
	\label{fig:asymptote}
\end{figure}
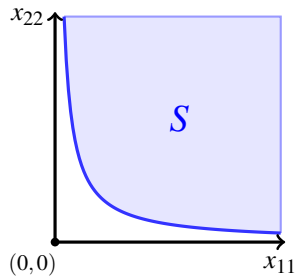

\end{example}

Weakly infeasible SDPs appear in the literature in many guises, some of which are modern and some others classic: 

\begin{itemize}
    \item they are difficult SDPs that are often mistaken for feasible ones by even the best solvers.
    
    \item they are closely related to linear maps under which the image of $\psdn$ is not closed. 
    Precisely,  \eqref{p} is weakly infeasible, if it is infeasible, but a sequence 
    $\{ X_i \} \subseteq \psdn \;  $ satisfies  
    $$\A X_i \rightarrow b, \; \text{as} \; i \rightarrow + \infty,$$ 
    in other words, when $b$ is in the  closure of $\A \psd{n}, \, $ but not in $\A \psd{n}$ itself. 
       Such linear maps cause other pathologies in SDPs, such as unattained optimal values and positive duality gaps \cite[Lemma 2]{Pataki:17}.  They  
    are also intriguing from the viewpoint of pure mathematics:  they  were  recently christened  {\em bad projections of the psd cone} by Jiang and Sturmfels \cite{jiang2020bad}  and explored from the perspective of algebraic geometry.

More broadly, bad projections are a good example of linear maps that carry a closed set into a nonclosed one, because  $\psdn$ is one of the simplest sets for which such maps even exist. The ``(non)closedness of the linear image" question appears  in several equivalent forms, for example, we may ask whether the sum of closed convex cones is closed. (Non)closedness of the linear image can be ensured by many diverse conditions:  some of the key ones are  
		Jameson's property (G) (Jameson \cite{jameson1972duality} and Bauschke et al. \cite{bauschke1999strong})
	and existence of certain tangent directions 
	\cite[Theorem 1.1, Theorem 5.1]{Pataki:07}.

\item they are {\em ill-posed}, i.e., their distance to the set of feasible instances is zero. Hence their infeasibility 
	cannot be detected by interior point methods  whose complexity depends on the distance to feasibility; 
		the best we   can do 
	is compute solutions of nearby feasible instances, see Pe{\~n}a and Renegar  \cite[Theorem 13]{pena2000computing}. Nor can we detect their infeasibility by the algorithm of Nesterov et al. in \cite{nesterov1999infeasible}, though 
	this algorithm can detect {\em near} ill-posedness.
	
	For a sample of the thriving literature on algorithm analysis based on 
	distance to (in)feasibility, we refer to 
	Renegar \cite{renegar1994some, renegar1995linear}; Pe{\~n}a  \cite{pena2000understanding}; and the comprehensive book by 
	B{\"u}rgisser and Cucker \cite{burgisser2013condition}.
				
	\item according to a classic viewpoint of  Klee  \cite{klee1961asymptotes}, when \eqref{p} is weakly infeasible, the affine subspace $\{ \, X \, : \, \A X = b \, \}$  is an {\em asymptote} of $\psd{n}.$ The asymptotic behavior is indeed apparent on Figure \ref{fig:asymptote}. \footnote{To be precise, although Klee introduced the notion of asymptotes, he did not specifically mention asymptotes of the psd cone.}
\end{itemize}

Although the infeasibility of  weakly infeasible SDPs cannot be reliably detected in general, several algorithmic approaches are available:

\begin{itemize}
	
		\item 	Facial reduction algorithms 
		can handle pathological SDPs, at least in theory,  since 
		they must be implemented in exact arithmetic. Facial reduction  originated in a paper by Borwein and Wolkowicz \cite{BorWolk:81}, then simpler variants 
		were proposed, for example  in \cite{pataki2000simple, Pataki:13}. We will use facial reduction as a theoretical tool 
		to develop our echelon form.
		In particular, we will use Parts (1) and (2) of Theorem 5 in \cite{liu2017exact}; an 
		earlier version of the latter 
		appeared in Louren{\c{c}}o et al. \cite{lourencco2016structural}.

	\item In contrast to the previous points, approximate, or robust solutions to SDPs in the Lasserre hierarchy of polynomial optimization can be  found by SDP solvers, even when exact solutions are impossible to compute:
	 see Henrion and Lasserre \cite{henrion2005detecting}
	and Lasserre and Magron \cite{lasserre2019sdp}. Some of these SDPs are weakly infeasible and one of them comes from minimizing the famous Motzkin polynomial. We closely examine this SDP in Example \ref{example-motzkin}. In other related work, the Douglas-Rachford method presented in Liu et al. \cite{liu2019new} successfully identified infeasibility of the weakly infeasible SDPs from \cite{liu2017exact}. 
\end{itemize}

To sketch our contributions, we revisit Example \ref{example-small}, where we naturally describe the affine subspace
\begin{equation} \label{eqn-H} 
H \, = \, \{ X  ~|~ X \, \text{is} ~ n ~ \text{by} ~ n ~ \text{symmetric}, ~ \A X = b \, \}
\end{equation}
in two ways. First, with equations $A_1 \bullet X = 0 \, $ and $A_2 \bullet X = 2, \, $ where
\begin{equation} \label{eqn-A1A2} 
A_1 = \begin{pmatrix} 1 & 0 \\ 0 & 0 \end{pmatrix} ~\text{and}~ A_2 = \begin{pmatrix} 0 & 1 \\ 1 & 0 \end{pmatrix}
\end{equation}
and the $\bullet$ product of symmetric matrices is the trace of their regular product. By the argument in Example \ref{example-small}, this representation certifies that \eqref{problem-small} is infeasible.

Besides,  $H = \{\lambda X_1 + X_2  ~:~ \lambda \in \rad{} \}$ where 
\begin{equation}  \label{eqn-X1X2}  
X_1 = \begin{pmatrix} 0 & 0 \\ 0 & 1 \end{pmatrix} ~ \text{and} ~ X_2 = \begin{pmatrix} 0 & 1 \\ 1 & 0 \end{pmatrix}. 
\end{equation} 
This generator representation  proves that $H$ is an asymptote of $\psd{2}, \, $ since  $ \lambda X_1 + X_2$ approaches $\psd{2}$ as   $\lambda \rightarrow + \infty.$

We see that $A_1, A_2$ and $X_1, X_2$ share a common ``echelon" structure and we may wonder whether such a structure  appears in every weakly infeasible SDP. The answer is naturally no, since we can easily ruin  this structure  even in \eqref{problem-small}. For example, we may take linear combinations of the equations  and perform congruence transformations,
in other words, replace both $A_i$ by $T^{\top} A_i T$ for some invertible $T.$

However,  it turns out that the  same operations can untangle  any weakly infeasible SDP. 
More precisely,  in Theorem \ref{thm-weak-ref} we develop an echelon form  of weakly infeasible SDPs with  the  following  features: i) it is constructed using elementary row operations and congruence transformations; ii)  it makes weak infeasibility evident, since  
the  matrices both in  the equality  and in the generator representation of the underlying affine subspace have the same echelon structure; and iii)  it permits us  to construct any weakly infeasible SDP.

Let us explain  the last point by an analogy with basic linear algebra.  We know that any infeasible linear system of equations $Ax = b$  can be brought to a  normal form 
\begin{equation} \label{eqn-Ax=b-normal} 
\begin{array}{rcl}
A^\prime x & = & b^\prime \\
0^{\top} x & = & 1
\end{array}
\end{equation}
using elementary row operations. Thus we can verify infeasibility of a linear system using the normal form \eqref{eqn-Ax=b-normal}. Further, we can construct any infeasible linear system as follows: we choose $A^\prime$ and $b^\prime$ in \eqref{eqn-Ax=b-normal} arbitrarily, then perform elementary row operations. This basic algorithm always succeeds and every infeasible linear system is among its outputs.

This work shows that  a similar scheme works for a more involved pathology -- weak infeasibility -- in a much more involved problem -- an SDP. 

Further, in Example \ref{example-motzkin} we present an SDP that is naturally in our echelon form, without ever having to perform elementary row operations or congruence transformations. This SDP arises from a sum-of-squares (SOS) relaxation of minimizing the famous Motzkin polynomial; we thus hope that our work will be of interest to the 
sum-of-squares optimization community. 

The plan of  the paper is as follows. In Section \ref{section-prelim} we review preliminaries consisting of basic linear algebra and SDP duality. In Section \ref{section-mainresults} we present and illustrate our main result, Theorem \ref{thm-weak-ref}, and to build intuition, we prove the ``easy" direction. In Section \ref{section-generate} we describe our algorithm to construct weakly infeasible SDPs and show that any weakly infeasible SDP is among its outputs. Our algorithm also constructs any bad projection of the psd cone. For the reader's convenience, some of the proofs are postponed to Section \ref{section-certificate-separately} and \ref{section-proof-thm1}. The most difficult proof is the ``hard" direction in Theorem \ref{thm-weak-ref}, which we give in Section \ref{section-proof-thm1}. Section \ref{section-computational} describes our problem library and computational tests. In Section \ref{section-discussion} we reinterpret Theorem \ref{thm-weak-ref} in two ways: as a ``sandwich" theorem and as a ``factorization" theorem. Here we also discuss open research directions.

To make the paper's results accessible to a broad audience, we prove them using only basic results in SDP duality and linear algebra, all of which we summarize in Section \ref{section-prelim}. This work has some unavoidable overlap with \cite{liu2017exact}, where the lemmas of Section \ref{section-certificate-separately} were already proved. On the other hand, here we prove these lemmas in a more elementary fashion. Reference \cite{liu2017exact} also gave a scheme to construct weakly infeasible SDPs in a certain restricted class; however, that scheme does not capture even some weakly infeasible SDPs with	$3 \times 3$ matrices. 
We comment in detail on these points in Section \ref{section-certificate-separately} and Section \ref{section-proof-thm1}.

\section{Preliminaries}
\label{section-prelim}

\paragraph{Operators and matrices}
For a linear operator (or matrix)  $\M, \,$ we denote its rangespace by $\R(\M)$ , its nullspace by $\N(\M)$ , and  its adjoint
by $\M^*.$

We denote the set of $n \times n$ symmetric matrices by $\sym{n}.$  Further, $N$ stands for the set $\{1,\dots,n\}$.

Given a matrix $M \in \mathbb{R}^{n \times n}$ and $R, S  \subseteq N, \,$ we denote the submatrix of $M$ corresponding to rows in $R$ and columns in $S$ by $M(R,S).$ When $R = \{r\}$ is a singleton, we simply write $M(r,S)$ for $M(\{r\}, S).$  For brevity, we let $M(R):=M(R,R).$

We denote the  concatenation of  matrices $A$ and $B$  along the diagonal by $A \oplus B, \, $
$$
A \oplus B := \begin{pmatrix} A & 0 \\
0  &  B \end{pmatrix}.
$$
Thus,  $M \oplus 0$  is the matrix obtained by attaching zero rows and colums to $M.$ The dimensions of $M \oplus 0$ will be clear from the context. 

Further,  $X \succeq 0$ means that the matrix $X$ is symmetric and positive semidefinite and 
$X \succ 0$ means it is symmetric and positive definite.

\paragraph{Basics of SDP duality}

Consider the pair of SDPs
\begin{center}
	\begin{minipage}{0.5\linewidth}
		\leqnomode
		\begin{equation}\label{p-opt}
		\begin{split}
		\hspace{0.75cm}\inf & ~~ C \bullet X \\
		\hspace{0.75cm}s.t.   & ~~  X ~  \text{is feasible in  } \eqref{p} \\
		\end{split}\tag{{P-\text{opt}}} 
		\end{equation}
	\end{minipage}%
	\begin{minipage}{0.5\linewidth}
		\begin{equation}\label{d}
		\begin{split}
		\sup  & ~~ b^{\top} y   \\
		s.t. & ~~  \A^*y \preceq C,   
			\end{split}\tag{D}
		\end{equation}
	\end{minipage}
\end{center}
where $C \in \symn,$ and for $T, S \in \symn$ we write $T \preceq S$ to say  $S - T \succeq 0. \,$ 
We say that \eqref{d} is the dual of \eqref{p-opt} and vice versa, \eqref{p-opt} is the dual of \eqref{d}.

When both are feasible, the optimal value of \eqref{p-opt} is at least as large as the optimal value of \eqref{d}. These optimal values agree and the optimal value of \eqref{p-opt} is attained when \eqref{d} satisfies {\em Slater's condition}, i.e., when there is  $y \in \rad{m}$  such that $C - \A^*y \succ 0.$

We say that \eqref{p} is {\em strongly infeasible}, if the distance of the affine subspace $H$ (see  \eqref{eqn-H}) from  $\psd{n}$ is positive. 
By this definition, we see that every infeasible SDP is either strongly or weakly infeasible.
We know that \eqref{p} is strongly infeasible exactly when its {\em alternative system}
\begin{equation} \label{p-alt} \tag{P-\text{alt}} 
\begin{array}{rcl}
\A^* y & \succeq & 0 \\
b^{\top} y & = & -1 
\end{array} 
\end{equation}
is feasible. In other words, strong infeasibility of \eqref{p} is certified by \eqref{p-alt}.

	  We will use the above results as building blocks, since  
	 they can be proven in a few pages,
	   with  real analysis and elementary linear algebra as  sole prerequisites; see Renegar
	\cite[Chapter 3]{Ren:01}.

\paragraph{Reformulations}

In the sequel we represent the operator $\A$ via symmetric matrices $A_1, \dots, A_m$ as 
\begin{equation} \label{eqn-repr-A}
\A X \, = \, (A_1 \bullet X, \dots, A_m \bullet X)^\top. \footnote{Then the adjoint $\A^*$ is given as $\A^* y = \sum_{i=1}^m y_i A_i$ for $y \in \rad{m}.$} 
\end{equation}
The following definition will be used  throughout the paper.

\begin{definition} \label{definition-reform} 
We say that we   reformulate \eqref{p} if we apply to it some of the following operations (in any order): 
\begin{enumerate}
    \item \label{exch} Exchange $(A_i, b_i)$ and $(A_j, b_j), \,$ where $i$ and $j$ are distinct indices in $\{1, \dots, m \}.$ 
    \item \label{trans} Replace $(A_i, b_i)$ by $\lambda (A_i, b_i) + \mu (A_j, b_j), \,$ where $\lambda$ and $\mu$ are reals, $\lambda \neq 0,$ 
    and $i$ and $j$ are distinct indices in $\{1, \dots, m \}.$ 
    \item \label{rotate} Replace all $A_i$ by $T^{\top} A_iT, \, $ where $T$ is a suitably chosen invertible matrix. 
\end{enumerate}
We also say that by reformulating (\ref{p}) we obtain a  reformulation; and that we reformulate the map $\A: \sym{n} \rightarrow \rad{m}$ if we reformulate \eqref{p} with some $b \in \rad{m}.$ 
\end{definition}

	We make two observations that will be useful later. 
	First, reformulating \eqref{p} preserves its status: it is feasible (infeasible, weakly infeasible) if and only if   has the same status after reformulating it. Second, in Definition \ref{definition-reform}  
operations \ref{exch} and \ref{trans} can be naturally viewed as  elementary row operations performed on the constraints of system $\A X = b.$ 
     
\paragraph{Semidefinite echelon form}

To motivate our next definition, suppose that a square matrix 
$M$ 
is in row echelon, i.e.,  upper triangular form 
$$
M \, =  \,  \begin{pmatrix}  * & * & \dots    & * \\
& *  & \dots   & * \\
&     & \ddots &   \\
&      &             & *
\end{pmatrix}
$$
where the empty cells are all zeroes. Assuming  the diagonal entries of $M$ are nonzero, this form serves two purposes. First,  it makes it clear that the columns of $M$ span the whole space; and second, it shows 
that the nullspace of $M$ contains only $0.$ 

Analogously, we define an echelon form of a sequence of symmetric matrices:
\begin{definition} \label{definition-regfr} We say that the sequence of symmetric $n \times n$ matrices $(M_1,\dots,M_k)$ is
in semidefinite echelon form with structure $\{P_1,\dots,P_{k}\}$ if  the following three conditions hold: i) the $P_i$ are disjoint subsets of $N, \,$ ii) for $i=1, \dots, k$  
\begin{equation}
\begin{array}{rcl} 
M_i(P_i) & \text{ is }   &  \text{diagonal with  positive  diagonal entries, and}   \\
M_i(P_1 \cup \dots \cup P_{i-1}, N) 	 & \text{ is } & \text{arbitrary},
\end{array}
\end{equation}
and iii) the remaining elements of all $M_i$ are zero.  (Note that by symmetry 
$M_i( N, P_1 \cup \dots \cup P_{i-1}) = M_i(P_1 \cup \dots \cup P_{i-1}, N)^\top.$) 

Thus, 
for a suitable permutation matrix $T$ the $M_i$ look like
\vspace{-0.4cm}
\begin{figure}[H] 
	\begin{center} 
		\begin{tikzpicture}[scale=0.5, every node/.style={scale=0.6}]
    
    %
    \newcommand{\tixmat}[3]{
    \foreach \x in {2,...,#3} \draw[] (\x-1+#1,#2) to (\x-1+#1,#2+#3);
    \foreach \y in {2,...,#3} \draw[] (#1,\y-1+#2) to (#1+#3,\y-1+#2);
    \draw[] (#1,#2) rectangle (#1+#3,#2+#3);
    }
    
    \fill[red!60] (0,3) rectangle (1,4);
    
    \fill[cyan!80] (1+5,3) rectangle (4+5,4);
    \fill[cyan!80] (0+5,0) rectangle (1+5,4);
    \fill[red!60] (1+5,2) rectangle (2+5,3);
    
    \fill[cyan!80] (2+10,2) rectangle (4+10,4);
    \fill[cyan!80] (0+10,0) rectangle (2+10,4);
    \fill[red!60] (2+10,1) rectangle (3+10,2);
    
    \tixmat{0}{0}{4};
    \tixmat{5}{0}{4};
    \tixmat{10}{0}{4};
    
    \node[] at (12,-0.75) {\LARGE $T^\top M_3T$};
    \node[] at (7,-0.75) {\LARGE $T^\top M_2T$};
    \node[] at (2,-0.75) {\LARGE $T^\top M_1T$};
    
    \draw[decoration={brace,raise=1.25pt},decorate] (0.05,4) -- node[above=6.75pt] {\Large $P_1^\prime$} (0.95,4);
    
    \draw[decoration={brace,raise=1.25pt},decorate] (5+1.05,4) -- node[above=6.75pt] {\Large $P_2^\prime$} (5+1.95,4);
    
    \draw[decoration={brace,raise=1.25pt},decorate] (10+2.05,4) -- node[above=6.75pt] {\Large $P_3^\prime$} (10+2.95,4);
    
    \node[] at (0.5,3.5) {\LARGE $+$};
    
    \node[] at (5+1.5,2.5) {\LARGE $+$};
    \node[] at (5+0.5,3.5) {\LARGE $\times$};
    \node[] at (5+1.5,3.5) {\LARGE $\times$};
    \node[] at (5+2.5,3.5) {\LARGE $\times$};
    \node[] at (5+3.5,3.5) {\LARGE $\times$};
    \node[] at (5+0.5,2.5) {\LARGE $\times$};
    \node[] at (5+0.5,1.5) {\LARGE $\times$};
    \node[] at (5+0.5,0.5) {\LARGE $\times$};
    
    \node[] at (10+2.5,1.5) {\LARGE $+$};
    \node[] at (10+0.5,3.5) {\LARGE $\times$};
    \node[] at (10+1.5,3.5) {\LARGE $\times$};
    \node[] at (10+2.5,3.5) {\LARGE $\times$};
    \node[] at (10+3.5,3.5) {\LARGE $\times$};
    \node[] at (10+0.5,2.5) {\LARGE $\times$};
    \node[] at (10+0.5,1.5) {\LARGE $\times$};
    \node[] at (10+0.5,0.5) {\LARGE $\times$};
    \node[] at (10+1.5,0.5) {\LARGE $\times$};
    \node[] at (10+1.5,1.5) {\LARGE $\times$};
    \node[] at (10+1.5,2.5) {\LARGE $\times$};
    \node[] at (10+2.5,2.5) {\LARGE $\times$};
    \node[] at (10+3.5,2.5) {\LARGE $\times$};
    
    \node[] at (15.5,2) {\LARGE $\cdot$};
    \node[] at (16,2) {\LARGE $\cdot$};
    \node[] at (16.5,2) {\LARGE $\cdot$};

    \end{tikzpicture}
	\end{center}
\end{figure}
\vspace{-0.7cm}
where the columns of $M_i$ with indices in $P_i$ were permuted into 
columns of $T^{\top} M_i T$ with indices in $P_i^\prime$.
\end{definition}

Here the $+$ red blocks are positive definite and diagonal, the $\times$ blue blocks may have arbitrary elements,
and the white blocks are zero.  

To highlight the analogy with the row echelon form, suppose $(M_1, \dots, M_k)$ is in semidefinite echelon form with 
structure $\{P_1, \dots, P_k\}.$ In Section \ref{section-mainresults} we will see that
this special form serves two purposes. First, suppose that $X \succeq 0$ satisfies $M_i \bullet X = 0$ for all $i.$ Then by an elementary argument the rows (and columns) of $X$ indexed by $P_1 \cup \dots \cup P_k$ are zero. Second, suppose $X$ is a symmetric matrix whose $N \setminus (P_1 \cup \dots \cup P_k)$ diagonal block is positive definite. Then by another elementary argument $\sum_{i=1}^k \lambda_i M_i + X$ is positive definite for suitable $\lambda_1, \dots, \lambda_k$ reals. 

\begin{remark}\normalfont
A sequence  $(M_1, \dots, M_k)$ in semidefinite echelon form is a type of {\em facial reduction sequence} \cite{BorWolk:81,pataki2000simple,Pataki:13}. Precisely, $(M_1, \dots, M_k)$ certifies that any $X \succeq 0$ such that 
$M_i \bullet X = 0$ for all $i$ belongs to a {\em face} of $\psdn, \, $ namely the set 
of of psd matrices with certain rows and columns equal to zero.
  
Let us consider a special case when the structure of $(M_1, \dots, M_k)$ is $\{P_1, \dots, P_k \}, \, $ 
$P_1$ contains the first $|P_1|$ indices of $N, \, $ $P_2$ contains the next $|P_2|$ indices, and so on, and the positive definite blocks in all $M_i$ are identities. These sequences were 
defined in \cite{liu2017exact}, and baptized  as 
{\em regularized facial reduction sequences}.   

\end{remark}

\section{The main result, and the easy direction} 
\label{section-mainresults} 

 The main result of the paper is the following.
 \begin{theorem}\label{thm-weak-ref} The problem 
 	\eref{p} is weakly infeasible if and only if it has a reformulation
 	\begin{equation}\label{equation-p-weak}\tag{$\mathrm{P}_{\rm weak}$}
 	\begin{array}{rcl} 
 	\A^\prime X & = & b^\prime \\  
 	X & \succeq & 0
 	\end{array}
 	\end{equation}
 	with the following  properties:
 	\begin{enumerate}
 		\item  \label{thm-weak-1} $(A_1^{\prime},\dots,A_{k+1}^\prime)$ is in semidefinite echelon form 
 		and $(b_1^\prime,\dots,b_k^\prime,b_{k+1}^\prime) = (0, \dots, 0, -1)$   for some $k \geq 1;$
 		\item  \label{thm-weak-2} there is 
 		$(X_1,\dots,X_{\ell+1})$ in semidefinite echelon form such that $\ell \geq 1$ and 
 		 \begin{equation}\label{eqn-Aiprime-Xj}
 		\begin{array}{rcll}
 		\mathcal{A}^{\prime}X_i &=& 0                          & \text{for} \,\, i=1,\dots,\ell\\
 		\mathcal{A}^{\prime}X_{\ell+1} &=& b^{\prime}.
 		\end{array}
 		\end{equation}
 		\end{enumerate}
 		 		\qed 
 \end{theorem}
 Here we understand that $\A^\prime$ is represented by 
symmetric matrices   $A_i^\prime$ as 
 \begin{equation} \label{eqn-repr-Aprime} 
 \A^\prime X \, = \, (A_1^\prime \bullet X, \dots, A_m^\prime \bullet X)^{\top}. 
 \end{equation}
 
	If \eqref{p} is weakly infeasible, then we can choose the reformulation \eqref{equation-p-weak} so the positive definite blocks in $A_1', \dots, A_k'$ and $X_1, \dots, X_\ell$ are all nonempty.
	Indeed, in the proof of the ``only if" direction we will  construct the 
	reformulation \eqref{equation-p-weak} and  $X_j$ sequence precisely in  this manner.

   \begin{example} \label{example-small-contd} (Example \ref{example-small} continued) 
 As a quick check, the problem \eqref{problem-small} needs only a minimal reformulation. To put it into the echelon form of \eqref{equation-p-weak},
 we set 
 \begin{equation} \label{eqn-A1primeA2prime} 
 A_1^\prime := A_1, \, A_2^\prime := - \frac{1}{2} A_2, 
 \end{equation}
  where $A_1$ and $A_2$ are given in 
 \eqref{eqn-A1A2}, and use $(X_1, X_2)$ from equation 
 \eqref{eqn-X1X2}. In this SDP we have $k = \ell =1.$ 
\end{example}

  It is useful to visualize Theorem \ref{thm-weak-ref} via  the matrix 
 of inner products of  the $A_i^\prime$ and $X_j$  in Figure \ref{figure-Ai-Xj-matrix}.
 \begin{figure}[H] 
 	\begin{center} 
 		\begin{tikzpicture}[scale=1.0, every node/.style={scale=1.0}]

    
    \begingroup
    \renewcommand*{\arraystretch}{1.5}
    \setlength\arraycolsep{5.5pt}
    
    \node[] at (0.45,0) 
    {
    $
    (A_i^\prime \bullet X_j)_{i=1, j=1}^{m, \ell+1} \, = \, \left( \begin{array}{ccc|c} 0 & \dots & 0    & 0 \\
 0 &    \ddots  & 0 & 0 \\
 0 & \dots       & 0 & 0  \\ \hline 
 0 & \dots       & 0 & -1   \\  \hline 
 0  & \dots      & 0 & b^\prime_{k+2}   \\
      & \ddots   &     &        \\
 0   & \dots     & 0 & b^\prime_m  \end{array}  \right) 
    $
    };
    
    \endgroup
    
    \draw[decoration={brace,raise=4.0pt},decorate] (0.35,2.3) -- node[above=7.5pt] {$\ell+1$} (3.05,2.3);
    
    \draw[decoration={brace,raise=4.0pt},decorate] (3.5,2.25) -- node[right=7.5pt] {$k+1$} (3.5,-0.25);
    
    \end{tikzpicture}
 	\end{center}
 	\caption{The matrix of  inner products of the $A_i^\prime$ and $X_j$ in Theorem \ref{thm-weak-ref}} 
 	\label{figure-Ai-Xj-matrix} 
 \end{figure}
The proof of the ``only if" direction of Theorem \ref{thm-weak-ref} is technical  and deferred to Section \ref{section-proof-thm1}. However, the proof of the  ``if" direction is elementary, and we provide it below.

\begin{proof}[Proof\;(of ``if'' in Theorem \ref{thm-weak-ref})]

It suffices to prove that \eqref{equation-p-weak}  is weakly infeasible. 
 To that end, we first prove it is  infeasible, so to obtain a contradiction 
 we  assume that $X$ is feasible in it. 
 We also assume that $(A_1^\prime, \dots, A_{k+1}^\prime)$ has structure $\{ P_1, \dots, P_{k+1}\}$
 (see Definition \ref{definition-regfr}).  
 
 Since $A_1^\prime \bullet X = 0, \,  $ a positively weighted linear combination of the diagonal elements of $X(P_1)$ is zero.
 Since $X \succeq 0, \, $  these elements are zero,  hence 
 the rows (and columns) of $X$ indexed by $P_1$ are zero.

 Continuing,   $A_2^\prime \bullet X = 0, \dots, A_k^\prime \bullet X = 0$ implies 
    that the rows (and columns) of $X$ indexed by 
  $P_2 \cup \dots \cup P_k \, $ are zero. Hence $A_{k+1}^\prime \bullet X$ is a positively weighted linear combination of the diagonal elements of $X(P_{k+1}), $ so 
 $$
 A_{k+1}^\prime \bullet X \geq 0.
 $$
  This contradiction proves that \eqref{equation-p-weak} is indeed infeasible.

  This process is illustrated on Figure \ref{figure-infeasible}, where the submatrices marked by $\oplus$ are positive semidefinite. For convenience 
  we assume in this figure that the columns of all matrices indexed by $P_1$ come first; the columns indexed by $P_2$ come next; etc.
      \begin{figure}[H]
  	\begin{center} 
  		\begin{tikzpicture}[scale=0.5, every node/.style={scale=0.6}]

\fill[red!60] (0,0) rectangle (4,4);
\draw (0,0) rectangle (4,4);
\node[] at (2,2) {\huge $\oplus$};

\draw[decoration={brace,raise=1.0pt,mirror},decorate] (0,-0.25) -- node[below=12pt] {\LARGE $X$} (4,-0.25);

\node[] at (6,2.75) {\LARGE $A_1^\prime\bullet X=0$};
\node[] at (6,2) {\Huge $\longrightarrow$};

\fill[red!60] (9,0) rectangle (12,3);
\draw[] (9,0) to (9,4);
\draw[] (8,3) to (12,3);
\draw (8,0) rectangle (12,4);
\node[] at (10.5,1.5) {\huge $\oplus$};
\node[] at (8.5,1.5) {\LARGE $0$};
\node[] at (8.5,3.5) {\LARGE $0$};
\node[] at (10.5,3.5) {\LARGE $0$};
\draw[decoration={brace,raise=1.25pt},decorate] (8.05,4) -- node[above=6.75pt] {\Large $P_1$} (8.95,4);
	 			
\draw[decoration={brace,raise=1.0pt,mirror},decorate] (8,-0.25) -- node[below=12pt] {\LARGE $X$} (12,-0.25);
	 			
\node[] at (14,2.75) {\LARGE $A_2^\prime\bullet X=0$};
\node[] at (14,2) {\Huge $\longrightarrow$};

\fill[red!60] (18,0) rectangle (20,2);
\draw[] (17,0) to (17,4);
\draw[] (16,3) to (20,3);
\draw[] (18,0) to (18,4);
\draw[] (16,2) to (20,2);
\draw (16,0) rectangle (20,4);
\node[] at (19,1) {\huge $\oplus$};
\node[] at (16.5,1) {\LARGE $0$};
\node[] at (17.5,1) {\LARGE $0$};
\node[] at (16.5,2.5) {\LARGE $0$};
\node[] at (17.5,2.5) {\LARGE $0$};
\node[] at (16.5,3.5) {\LARGE $0$};
\node[] at (17.5,3.5) {\LARGE $0$};
\node[] at (19,3.5) {\LARGE $0$};
\node[] at (19,2.5) {\LARGE $0$};
\draw[decoration={brace,raise=1.25pt},decorate] (16.05,4) -- node[above=6.75pt] {\Large $P_1$} (16.95,4);
\draw[decoration={brace,raise=1.25pt},decorate] (17.05,4) -- node[above=6.75pt] {\Large $P_2$} (17.95,4);
	 			
\draw[decoration={brace,raise=1.0pt,mirror},decorate] (16,-0.25) -- node[below=12pt] {\LARGE $X$} (20,-0.25);

\node[] at (22,2.75) {\LARGE $\dots$};
\node[] at (22,2) {\Huge $\longrightarrow$};

\end{tikzpicture}
  	\end{center}
  	\caption{Proving that \eqref{equation-p-weak} is infeasible} 
  	\label{figure-infeasible} 
  \end{figure}
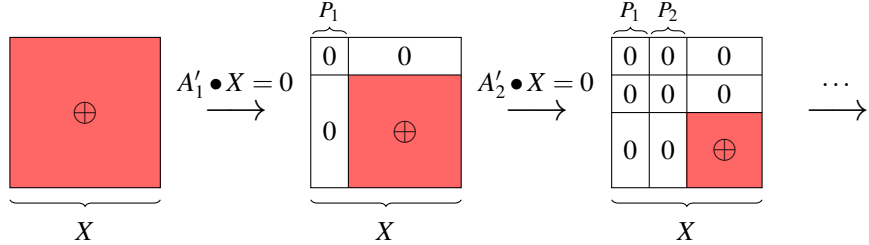 
     Next we prove that \eqref{equation-p-weak} is not strongly infeasible. For that, let 
$$
H^\prime = \{ \, X\in \symn   ~ : ~  \A^\prime X = b^\prime \  \},
$$
and  fix $\epsilon > 0. \, $  

We will construct a psd matrix which is $\epsilon$ close to $H^\prime.$ 
	Suppose  the structure of 
	$(X_1, \dots, X_{\ell+1})$ is $\{Q_1, \dots, Q_{\ell+1} \}$ and for brevity, let $Q_{\ell + 2} = N \setminus (Q_1 \cup \dots \cup Q_{\ell+1}).$ 
	
	First we  define $X_\delta  \in \symn$ so that $X_\delta(Q_{\ell+2}) = \delta I$ 
	and the other elements  of $X_\delta$ are zero. Here 
	$\delta >0$ is chosen so  the norm of 
	$X_\delta$ is at most $\epsilon.$ See the leftmost picture in Figure \ref{figure-not-strong}. 
	
	Second, we define $X_{\ell+1}^\prime := X_{\ell+1} + X_\delta.$ Then the 
	$(Q_{\ell+1} \cup Q_{\ell+2})$ diagonal block of $X_{\ell+1}^\prime$  is 
	positive definite and $X_{\ell+1}^\prime$ is within $\epsilon$ distance of $H^\prime$ (since $X_{\ell+1} \in H^\prime$). See the middle picture in Figure \ref{figure-not-strong}. 
	
	Next we let 
	$X_\ell^\prime := \gamma_\ell X_\ell + X_{\ell+1}^\prime$ 
	where $\gamma_\ell$ is a positive real. 
	The definition of positive definiteness ($G \succ 0$ 
	if $x^{\top} G x  > 0$ for all nonzero $x$) implies that 
	$$X_\ell^\prime(Q_\ell \cup Q_{\ell+1} \cup Q_{\ell+2}) \succ 0$$ 
	if $\gamma_\ell$ is sufficiently large. Further,  $X_\ell^\prime$ is still within $\epsilon$ distance of  $H^\prime$
	(since $\A' X_{\ell} = 0$ ). 
	We refer to the rightmost picture in Figure \ref{figure-not-strong}. 

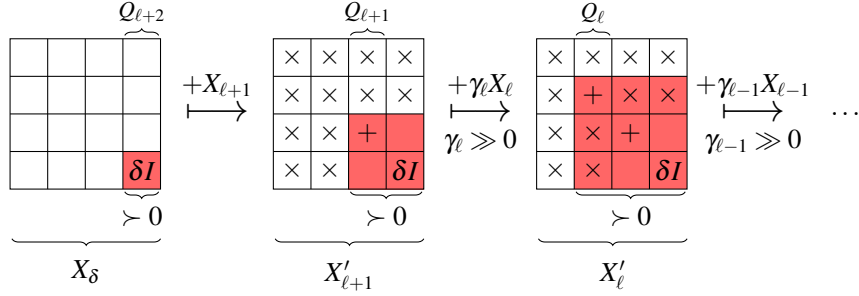
\begin{figure}[H]
	\begin{center} 
		\begin{tikzpicture}[scale=0.5, every node/.style={scale=0.6}]
    
    %
    \newcommand{\tixmat}[3]{
    \foreach \x in {2,...,#3} \draw[] (\x-1+#1,#2) to (\x-1+#1,#2+#3);
    \foreach \y in {2,...,#3} \draw[] (#1,\y-1+#2) to (#1+#3,\y-1+#2);
    \draw[] (#1,#2) rectangle (#1+#3,#2+#3);
    }
    
    \fill[red!60] (3,0) rectangle (4,1);
    
    \tixmat{0}{0}{4};
    
    \draw[decoration={brace,raise=1.0pt,mirror},decorate] (3.05,0) -- node[below=7.5pt] {\LARGE $\succ0$} (3.95,0);
    
    \node[] at (3.5,0.5) {\LARGE $\delta I$};

    \node[] at (5.5,2) {\Huge $\longmapsto$};
    \node[] at (5.5,2.75) {\LARGE $+X_{\ell+1}$};

    \fill[red!60] (2+7,0) rectangle (4+7,2);
    
    \tixmat{7}{0}{4};
    
    \draw[decoration={brace,raise=1.0pt,mirror},decorate] (7+2.05,0) -- node[below=7.5pt] {\LARGE $\succ0$} (7+3.95,0);
    
    \node[] at (7+3.5,0.5) {\LARGE $\delta I$};
    \node[] at (7+2.5,1.5) {\LARGE $+$};
    \node[] at (7+0.5,3.5) {\LARGE $\times$};
    \node[] at (7+1.5,3.5) {\LARGE $\times$};
    \node[] at (7+2.5,3.5) {\LARGE $\times$};
    \node[] at (7+3.5,3.5) {\LARGE $\times$};
    \node[] at (7+0.5,2.5) {\LARGE $\times$};
    \node[] at (7+0.5,1.5) {\LARGE $\times$};
    \node[] at (7+0.5,0.5) {\LARGE $\times$};
    \node[] at (7+1.5,0.5) {\LARGE $\times$};
    \node[] at (7+1.5,1.5) {\LARGE $\times$};
    \node[] at (7+1.5,2.5) {\LARGE $\times$};
    \node[] at (7+2.5,2.5) {\LARGE $\times$};
    \node[] at (7+3.5,2.5) {\LARGE $\times$};

    \node[] at (12.5,2) {\Huge $\longmapsto$};
    \node[] at (12.5,2.75) {\LARGE $+\gamma_{\ell} X_{\ell}$};
    \node[] at (12.5,1.25) {\LARGE $\gamma_{\ell}\gg0$};

    \fill[red!60] (1+14,0) rectangle (4+14,3);
    
    \tixmat{14}{0}{4};
    
    \draw[decoration={brace,raise=1.0pt,mirror},decorate] (14+1.05,0) -- node[below=7.5pt] {\LARGE $\succ0$} (14+3.95,0);
    
    \node[] at (14+1.5,2.5) {\LARGE $+$};
    \node[] at (14+3.5,0.5) {\LARGE $\delta I$};
    \node[] at (14+2.5,1.5) {\LARGE $+$};
    \node[] at (14+0.5,3.5) {\LARGE $\times$};
    \node[] at (14+1.5,3.5) {\LARGE $\times$};
    \node[] at (14+2.5,3.5) {\LARGE $\times$};
    \node[] at (14+3.5,3.5) {\LARGE $\times$};
    \node[] at (14+0.5,2.5) {\LARGE $\times$};
    \node[] at (14+0.5,1.5) {\LARGE $\times$};
    \node[] at (14+0.5,0.5) {\LARGE $\times$};
    \node[] at (14+1.5,0.5) {\LARGE $\times$};
    \node[] at (14+1.5,1.5) {\LARGE $\times$};
    \node[] at (14+2.5,2.5) {\LARGE $\times$};
    \node[] at (14+3.5,2.5) {\LARGE $\times$};

    \node[] at (19.75,2) {\Huge $\longmapsto$};
    \node[] at (19.75,2.75) {\LARGE $+\gamma_{\ell-1} X_{\ell-1}$};
    \node[] at (19.75,1.25) {\LARGE $\gamma_{\ell-1}\gg0$};
    
    \node[] at (20+2.25,2) {\LARGE $\dots$};

    \draw[decoration={brace,raise=1.0pt,mirror},decorate] (0+0.05,-1.25) -- node[below=12pt] {\LARGE $X_\delta$} (4,-1.25);
    
    \draw[decoration={brace,raise=1.0pt,mirror},decorate] (7+0.05,-1.25) -- node[below=12pt] {\LARGE $X_{\ell+1}^\prime$} (11,-1.25);
    
    \draw[decoration={brace,raise=1.0pt,mirror},decorate] (14+0.05,-1.25) -- node[below=12pt] {\LARGE $X_\ell^\prime$} (18,-1.25);
    
    \draw[decoration={brace,raise=1.25pt},decorate] (3.05,4) -- node[above=6.75pt] {\Large $Q_{\ell+2}$} (3.95,4);
    
    \draw[decoration={brace,raise=1.25pt},decorate] (9.05,4) -- node[above=6.75pt] {\Large $Q_{\ell+1}$} (9.95,4);
    
    \draw[decoration={brace,raise=1.25pt},decorate] (15.05,4) -- node[above=6.75pt] {\Large $Q_\ell$} (15.95,4);

    \end{tikzpicture}
	\end{center}
	\caption{Proving that \eqref{equation-p-weak} is not strongly infeasible} 
	\label{figure-not-strong} 
\end{figure} 
Continuing  in this fashion we add $\gamma_{\ell-1} X_{\ell-1}$ to $X_\ell^\prime$ for some large $\gamma_{\ell-1}$ and so on. Eventually we obtain a positive definite matrix, within $\epsilon$ distance of $H^\prime, \, $ and conclude that 
\eqref{equation-p-weak} is not strongly infeasible. The proof is complete.

\end{proof}

\begin{remark} \label{remark-HT} \normalfont
	Suppose \eqref{p} is weakly infeasible.
	Based on Theorem \ref{thm-weak-ref} we can prove this to a ``third party" 
	by the following data:
	\begin{enumerate}
		\item The original problem \eqref{p} and the reformulation \eqref{equation-p-weak};
		\item The sequence of matrices $(X_1, \dots, X_{\ell+1});$ 
		\item \label{item-operations} The operations needed to reformulate \eqref{p}  into \eqref{equation-p-weak}. These can be encoded 
		in a very compact manner, just by two matrices: 
		the elementary row operations by an $m \times m$ matrix $G = (g_{ij})$ 
		and the congruence transformations by an $n \times n$ matrix $T.$ Then the equations
		\begin{equation} \label{eqn-Aip-bip} 
		\begin{array}{rcl} 
		A_i^\prime & = &  T^{\top} \bigl( \sum_{j=1}^m g_{ij} A_j  \bigr) T \quad \text{for} \quad i=1, \dots, m \\
		b^\prime  &  = &  Gb
		\end{array} 
		\end{equation}
		hold. 
	\end{enumerate}
	So, to verify that \eqref{p} is weakly infeasible, we check that \eqref{equation-p-weak} and
	$(X_1, \dots, X_{\ell+1})$ are in the required form,  and equations \eqref{eqn-Aiprime-Xj} and 
	\eqref{eqn-Aip-bip} hold. All  these  
	computations must be done  over real numbers and in  exact arithmetic.
		This discussion implies that the problem ``is \eqref{p} weakly infeasible?" is in 
	$\np$   
	in the real number model of computing \footnote{It is also in $\conp, $ since if \eqref{p} is not weakly infeasible, we can verify this by exhibiting either a feasible solution of \eqref{p}, or a feasible solution of \eqref{p-alt}. }.
	
	This result already follows from previous works  \cite{Ramana:97, KlepSchw:12}. Precisely, these papers  described  certificates  to verify  infeasibility of any infeasible SDP,  regardless of whether it is strongly or  weakly infeasible. For example, Ramana's infeasibility certificate of an SDP is a semidefinite system which is feasible exactly when the SDP in question is infeasible.  Note that  \eqref{equation-p-infeas} in Lemma \ref{thm-inf} also certifies infeasibility of \eqref{p}.
	Thus, using any of these certificates, we obtain a certificate that \eqref{p} is {\em weakly} infeasible, if we  verify that \eqref{p} and its alternative system \eqref{p-alt} are  {\em both} infeasible. 
	
	On the other hand, our  echelon form \eqref{equation-p-weak} does more than just verify weak infeasibility. It makes weak infeasibility evident to see; and we can use it to conveniently construct any weakly infeasible SDP,  a feature that previously known certificates do not have.   
	
	Note that the following related question: ``Can we decide the feasibility status (feasibility, or weak/strong  infeasibility)  of \eqref{p} in polynomial time?" 
	is open   in the real number model, and in the Turing model as well.

\end{remark}

\begin{example} \label{example-large} 
 The SDP in the form \eqref{p} with data 
\begin{equation} \label{eqn-weak-large} 
\begin{array}{ccccccccccccccccccc}
A_1 & = & \begin{pmatrix}
	8 & -1 & -9 & -2 \\
	-1 & -26 & 3 & 39 \\
	-9 & 3 & 10 & 3 \\
	-2 & 39 & 3 & -16
\end{pmatrix}, \, 
A_2 & = & \begin{pmatrix}
	5 & -3 & -6 & -2 \\
	-3 & -6 & 5 & 21 \\
	-6 & 5 & 7 & 2 \\
	-2 & 21 & 2 & -11
\end{pmatrix} \\ \\
A_3 & = & \begin{pmatrix}
	-6 & -3 & 7 & 4 \\
	-3 & 34 & 1 & -43 \\
	7 & 1 & -8 & -5 \\
	4 & -43 & -5 & 18
\end{pmatrix}, \,
A_4 & = & \begin{pmatrix}
	5 & 4 & -9 & -6 \\
	4 & -28 & 6 & 48 \\
	-9 & 6 & 13 & 5 \\
	-6 & 48 & 5 & -21
\end{pmatrix} \\ \\
b  & = & \!\!\!\!\!\!\!\!\!\!   (-44, -22, 44, -68)^{\top} 
\end{array}
\end{equation} 
is weakly infeasible, but from this form this would be very difficult to tell.

However, once we reformulate \eqref{eqn-weak-large}  by the formulas in 
\eqref{eqn-Aip-bip} and the $G$ and $T$ matrices 
$$
G = \dfrac{1}{2} \begin{pmatrix}
1 & 0 & 1 & 0 \\
0 & 2 & 1 & 0 \\
1 & 1 & 3 & 1 \\
0 & 0 & 1 & 1
\end{pmatrix}, \quad 
T=\begin{pmatrix}
-1 & 1 & 1 & 1 \\
0 & 1 & 0 & 0 \\
0 & -1 & 0 & 1 \\
0 & 0 & -1 & 0
\end{pmatrix}, 
$$
it is brought into the form \eqref{equation-p-weak} with the $A_i^\prime$ and $X_j$ shown on Figure 
\ref{figure-large-weak}. 
\begin{figure}[H]
	\begin{center} 
		\begin{tikzpicture}[scale=0.5, every node/.style={scale=0.6}]
    
    %
    \newcommand{\tixmat}[3]{
    \foreach \x in {2,...,#3} \draw[] (\x-1+#1,#2) to (\x-1+#1,#2+#3);
    \foreach \y in {2,...,#3} \draw[] (#1,\y-1+#2) to (#1+#3,\y-1+#2);
    \draw[] (#1,#2) rectangle (#1+#3,#2+#3);
    }
    
    \fill[red!60] (0,3) rectangle (1,4);
    
    \fill[cyan!80] (5+0,0) rectangle (5+1,4);
    \fill[cyan!80] (5+1,3) rectangle (5+4,4);
    \fill[red!60] (5+1,2) rectangle (5+2,3);
    
    \fill[cyan!80] (10+0,0) rectangle (10+2,4);
    \fill[cyan!80] (10+2,2) rectangle (10+4,4);
    \fill[red!60] (10+2,1) rectangle (10+3,2);
    
    \fill[red!60] (3,-6) rectangle (4,-6+1);
    
    \fill[cyan!80] (5+0,-6+0) rectangle (5+4,-6+1);
    \fill[cyan!80] (5+3,-6+1) rectangle (5+4,-6+4);
    \fill[red!60] (5+1,-6+2) rectangle (5+2,-6+3);
    
    \fill[cyan!80] (5*2+0,-6+0) rectangle (5*2+4,-6+1);
    \fill[cyan!80] (5*2+0,-6+2) rectangle (5*2+4,-6+3);
    \fill[cyan!80] (5*2+1,-6+0) rectangle (5*2+2,-6+4);
    \fill[cyan!80] (5*2+3,-6+0) rectangle (5*2+4,-6+4);
    \fill[red!60] (5*2+2,-6+1) rectangle (5*2+3,-6+2);
    
    \tixmat{0}{0}{4};
    \tixmat{5}{0}{4};
    \tixmat{10}{0}{4};
    
    \tixmat{0}{-6}{4};
    \tixmat{5}{-6}{4};
    \tixmat{10}{-6}{4};
    
    \node[] at (2+5*2,-0.75) {\LARGE $A_3^\prime$};
    \node[] at (2+5,-0.75) {\LARGE $A_2^\prime$};
    \node[] at (2,-0.75) {\LARGE $A_1^\prime$};
    
    \node[] at (2+5*2,-0.75-6) {\LARGE $X_3$};
    \node[] at (2+5,-0.75-6) {\LARGE $X_2$};
    \node[] at (2,-0.75-6) {\LARGE $X_1$};
    
    \draw[decoration={brace,raise=1.25pt},decorate] (0,3.05) -- node[left=6.75pt] {\Large $P_1$} (0,3.95);
    \draw[decoration={brace,raise=1.25pt},decorate] (0,2.05) -- node[left=6.75pt] {\Large $P_2$} (0,2.95);
    \draw[decoration={brace,raise=1.25pt},decorate] (0,1.05) -- node[left=6.75pt] {\Large $P_3$} (0,1.95);
    
    \draw[decoration={brace,raise=1.25pt},decorate] (0,0.05-6) -- node[left=6.75pt] {\Large $Q_1$} (0,0.95-6);
    \draw[decoration={brace,raise=1.25pt},decorate] (0,2.05-6) -- node[left=6.75pt] {\Large $Q_2$} (0,2.95-6);
    \draw[decoration={brace,raise=1.25pt},decorate] (0,1.05-6) -- node[left=6.75pt] {\Large $Q_3$} (0,1.95-6);
    
    \node[] at (0.5,3.5) {\LARGE $1$};
    
    \node[] at (5+2.5,3.5) {\LARGE $-2$};
    \node[] at (5+3.5,3.5) {\Large $1/2$};
    \node[] at (5+0.5,1.5) {\LARGE $-2$};
    \node[] at (5+0.5,0.5) {\Large $1/2$};
    \node[] at (5+1.5,2.5) {\LARGE $1$};
    \node[] at (5+0.5,2.5) {\LARGE $0$};
    \node[] at (5+0.5,3.5) {\LARGE $2$};
    \node[] at (5+1.5,3.5) {\LARGE $0$};
    
    \node[] at (10+3.5,2.5) {\LARGE $1$};
    \node[] at (10+1.5,0.5) {\LARGE $1$};
    \node[] at (10+2.5,1.5) {\LARGE $1$};
    \node[] at (10+0.5,0.5) {\Large $3/2$};
    \node[] at (10+0.5,1.5) {\LARGE $1$};
    \node[] at (10+0.5,2.5) {\LARGE $3$};
    \node[] at (10+0.5,3.5) {\LARGE $0$};
    \node[] at (10+1.5,3.5) {\LARGE $3$};
    \node[] at (10+2.5,3.5) {\LARGE $1$};
    \node[] at (10+3.5,3.5) {\Large $3/2$};
    \node[] at (10+2.5,2.5) {\LARGE $4$};
    \node[] at (10+1.5,1.5) {\LARGE $4$};
    \node[] at (10+1.5,2.5) {\LARGE $1$};
    
    \node[] at (3.5,0.5-6) {\LARGE $1$};
    
    \node[] at (1+5.5,2+0.5-6) {\LARGE $1$};
    \node[] at (0+5.5,0+0.5-6) {\LARGE $-1$};
    \node[] at (1+5.5,0+0.5-6) {\LARGE $1$};
    \node[] at (2+5.5,0+0.5-6) {\LARGE $0$};
    \node[] at (3+5.5,0+0.5-6) {\LARGE $-4$};
    \node[] at (3+5.5,1+0.5-6) {\LARGE $0$};
    \node[] at (3+5.5,2+0.5-6) {\LARGE $1$};
    \node[] at (3+5.5,3+0.5-6) {\LARGE $-1$};
    
    \node[] at (2+10.5,1+0.5-6) {\LARGE $1$};
    \node[] at (0+10.5,0+0.5-6) {\LARGE $0$};
    \node[] at (1+10.5,0+0.5-6) {\LARGE $-5$};
    \node[] at (2+10.5,0+0.5-6) {\LARGE $-4$};
    \node[] at (3+10.5,0+0.5-6) {\LARGE $2$};
    \node[] at (3+10.5,1+0.5-6) {\LARGE $-4$};
    \node[] at (3+10.5,2+0.5-6) {\LARGE $-5$};
    \node[] at (3+10.5,3+0.5-6) {\LARGE $0$};
    \node[] at (2+10.5,2+0.5-6) {\LARGE $-2$};
    \node[] at (1+10.5,2+0.5-6) {\LARGE $0$};
    \node[] at (0+10.5,2+0.5-6) {\LARGE $4$};
    \node[] at (1+10.5,3+0.5-6) {\LARGE $4$};
    \node[] at (1+10.5,1+0.5-6) {\LARGE $-2$};
    
    \end{tikzpicture}.
	\end{center}
	\label{fig-A1A2A3} 
	\caption{The $A_i'$ and $X_j$ obtained from reformulating   the SDP \eqref{eqn-weak-large}} 
	\label{figure-large-weak} 
\end{figure}
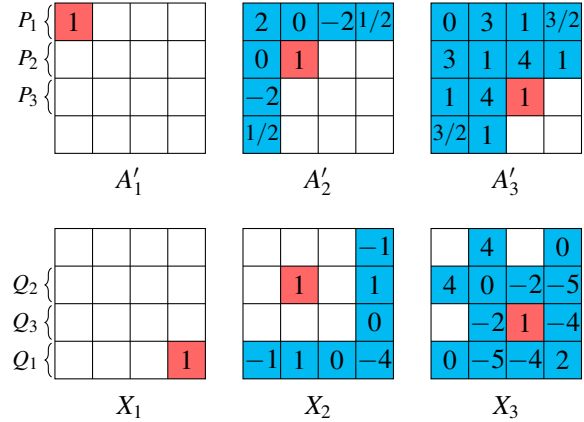
In the reformulation equations $A_1^\prime \bullet X = A_2^\prime \bullet X = 0, \, A_3^\prime \bullet X = -1$  certify  infeasibility and $(X_1, X_2,X_3)$ certify not strong infeasibility. The matrix $A_4^\prime$ is omitted 
(and is straightforward to compute from the formulas in \eqref{eqn-Aip-bip}). Note that now $k = \ell = 2.$ 
\end{example}

The reader may ask whether some SDPs are naturally in the echelon form of \eqref{equation-p-weak} 
{\em without even having to reformulate them.} We next present such an SDP from a 
prominent application of semidefinite programming, polynomial optimization. 

We first recall a definition. Given a multivariate polynomial 
$f = f(x_1, \dots, x_n), \, $ we say that $f$ is a sum of squares (SOS) if 
$f = \sum_{i=1}^t f_i^2$ for some $t$ positive integer and $f_i$ polynomials. 
An SOS  polynomial  is of course nonnegative. 
On the other hand, the first example of a nonnegative, but not SOS polynomial  was given 
by Motzkin in \cite{motzkin1967arithmetic} and there are many more nonnegative polynomials than SOS polynomials: see Blekherman  \cite{blekherman2006there}.

\begin{example} \label{example-motzkin} 
		
	Given the famous Motzkin polynomial
	\begin{equation}
	f(x,y) = 1 - 3 x^2 y^2 + x^2 y^4 + x^4 y^2, 
	\end{equation}
we can find its infimum over $\rad{2}$ by solving the problem 
	\begin{equation} \label{eqn-motzkin-exact} 
	\begin{array}{rlcl} 
	\sup & \lambda \\
	s.t.   & f(x,y)  - \lambda & \geq & 0.
	\end{array}
	\end{equation}
	 In the SOS relaxation of  \eqref{eqn-motzkin-exact} proposed in  \cite{lasserre2001global} and 
	 \cite{parrilo2003semidefinite}  we  solve the following problem instead: 
	 \begin{equation} \label{eqn-motzkin-sos} 
	\begin{array}{rlcl} 
	\sup & \lambda \\
	s.t.   & f - \lambda & \text{is} & \text{SOS}. 
	\end{array}
	 \end{equation}
	  In turn, we formulate \eqref{eqn-motzkin-sos} as an SDP as follows. 
	 We define a vector 
	 of monomials  \footnote{To strictly follow the SOS recipe we should also include in $z$ the monomials $x^3$ and $y^3.$ 
	 	We omitted these for simplicity, but it is straightforward to check that even if we do  include them, the resulting SDP 
	 	is still in the echelon form of \eqref{equation-p-weak}.} 
	 $$
	z = (x^2, y^2, x, y, xy, xy^2, x^2y, 1)^\top,
	$$
	then we know that $f - \lambda$ is SOS if and only if
$f - \lambda  = X \bullet zz^\top$ for some $X \succeq 0.$ 
	
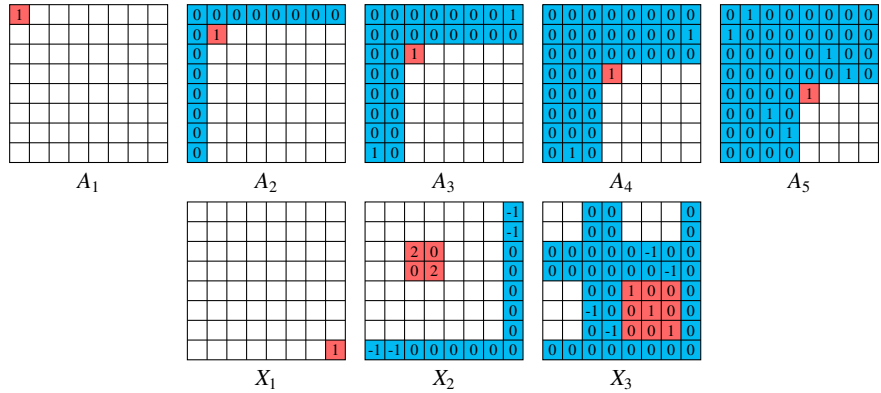
\begin{figure}[htp]  
	\begin{center} 
\begin{tikzpicture}[scale=0.35*0.75, every node/.style={scale=0.35*1}]
    
    %
    \newcommand{\tixmat}[3]{
    \foreach \x in {2,...,#3} \draw[] (\x-1+#1,#2) to (\x-1+#1,#2+#3);
    \foreach \y in {2,...,#3} \draw[] (#1,\y-1+#2) to (#1+#3,\y-1+#2);
    \draw[] (#1,#2) rectangle (#1+#3,#2+#3);
    }
    
    \newcommand{\tixzeros}[3]{
    
    \foreach \x in {1,...,#3} \foreach \y in {1,...,#3} \node[] at (\x-0.5+#1,\y-0.5+#2) {\Large 0};

    }
    
    \newcommand{\blrd}[6]{
    
    \fill[#5] (#1,#2) rectangle (#3+#1,#4+#2);
    
    \foreach \x in {1,...,#3} \foreach \y in {1,...,#4} \node[] at (\x-0.5+#1,\y-0.5+#2) {\LARGE #6};

    }
    

    \blrd{0}{7}{1}{1}{red!60}{1};
    
    \blrd{9}{0}{1}{8}{cyan!80}{0};
    \blrd{9}{7}{8}{1}{cyan!80}{0};
    \blrd{9+1}{6}{1}{1}{red!60}{1};
    
    \blrd{9*2}{0}{2}{8}{cyan!80}{0};
    \blrd{9*2}{6}{8}{2}{cyan!80}{0};
    \blrd{9*2+2}{5}{1}{1}{red!60}{1};
    \blrd{9*2}{0}{1}{1}{cyan!80}{1};
    \blrd{9*2+7}{7}{1}{1}{cyan!80}{1};
    
    \blrd{9*3}{0}{3}{8}{cyan!80}{0};
    \blrd{9*3}{5}{8}{3}{cyan!80}{0};
    \blrd{9*3+3}{4}{1}{1}{red!60}{1};
    \blrd{9*3+1}{0}{1}{1}{cyan!80}{1};
    \blrd{9*3+7}{6}{1}{1}{cyan!80}{1};
    
    \blrd{9*4}{0}{4}{8}{cyan!80}{0};
    \blrd{9*4}{4}{8}{4}{cyan!80}{0};
    \blrd{9*4+4}{3}{1}{1}{red!60}{1};
    \blrd{9*4+0}{6}{1}{1}{cyan!80}{1};
    \blrd{9*4+1}{7}{1}{1}{cyan!80}{1};
    \blrd{9*4+2}{2}{1}{1}{cyan!80}{1};
    \blrd{9*4+3}{1}{1}{1}{cyan!80}{1};
    \blrd{9*4+5}{5}{1}{1}{cyan!80}{1};
    \blrd{9*4+6}{4}{1}{1}{cyan!80}{1};

    \blrd{7+9}{-10}{1}{1}{red!60}{1};
    
    \blrd{9+9}{-10}{8}{1}{cyan!80}{0};
    \blrd{9+7+9}{-10}{1}{8}{cyan!80}{0};
    \blrd{2+9+9}{-10+4}{2}{2}{red!60}{0};
    \blrd{2+9+9}{-10+5}{1}{1}{red!60}{2};
    \blrd{3+9+9}{-10+4}{1}{1}{red!60}{2};
    \blrd{9+9}{-10}{2}{1}{cyan!80}{-1};
    \blrd{9+7+9}{-10+6}{1}{2}{cyan!80}{-1};
    
    \blrd{0+9*2+9}{-10}{8}{1}{cyan!80}{0};
    \blrd{7+9*2+9}{-10}{1}{8}{cyan!80}{0};
    \blrd{2+9*2+9}{-10}{2}{8}{cyan!80}{0};
    \blrd{0+9*2+9}{-10+4}{8}{2}{cyan!80}{0};
    \blrd{4+9*2+9}{-10+1}{3}{3}{red!60}{0};
    \blrd{4+9*2+9}{-10+3}{1}{1}{red!60}{1};
    \blrd{5+9*2+9}{-10+2}{1}{1}{red!60}{1};
    \blrd{6+9*2+9}{-10+1}{1}{1}{red!60}{1};
    \blrd{2+9*2+9}{-10+2}{1}{1}{cyan!80}{-1};
    \blrd{3+9*2+9}{-10+1}{1}{1}{cyan!80}{-1};
    \blrd{5+9*2+9}{-10+5}{1}{1}{cyan!80}{-1};
    \blrd{6+9*2+9}{-10+4}{1}{1}{cyan!80}{-1};

    \node[] at (4+9*4,-1) {\Huge $A_5$};
    \node[] at (4+9*3,-1) {\Huge $A_4$};
    \node[] at (4+9*2,-1) {\Huge $A_3$};
    \node[] at (4+9,-1) {\Huge $A_2$};
    \node[] at (4,-1) {\Huge $A_1$};
    
    \node[] at (4+9*3,-11) {\Huge $X_3$};
    \node[] at (4+9*2,-11) {\Huge $X_2$};
    \node[] at (4+9,-11) {\Huge $X_1$};

    \tixmat{0}{0}{8};
    \tixmat{9}{0}{8};
    \tixmat{9*2}{0}{8};
    \tixmat{9*3}{0}{8};
    \tixmat{9*4}{0}{8};
    
    \tixmat{0+9}{-10}{8};
    \tixmat{9+9}{-10}{8};
    \tixmat{9*2+9}{-10}{8};
    
    \end{tikzpicture}  
	\end{center}
	\caption{Certificates of weak infeasibility in the Motzkin polynomial SDP} 
	\label{figure-motzkin} 
\end{figure}

	We then match the coefficients of monomials in $f-\lambda$ and $X \bullet zz^\top$ and obtain the SDP
	
	\begin{equation} \label{eqn-E11bulletQ} 
	\begin{array}{rrcll} 
	\sup &  - E_{88} \bullet X \\
	s.t. & E_{11} \bullet X & =  & 0 & (x^4) \\
	& E_{22} \bullet X  & = & 0 & (y^4) \\
	& (E_{33} +  E_{18}) \bullet X  & = & 0 & (x^2) \\
	& (E_{44} +  E_{28}) \bullet X  & = & 0 & (y^2) \\
	& (E_{55} +  E_{12} +  E_{36} +  E_{47})  \bullet X  & = & -3 & (x^2 y^2) \\
	& & \vdots & \\
	&  X          & \succeq & 0.  
	\end{array} 
	\end{equation}
		In \eqref{eqn-E11bulletQ} the $E_{ij}$ are unit matrices in $\sym{8}$ whose elements in the $(i,j)$ and $(j,i)$ position are $1$ and the rest zero. For each equation we show the corresponding monomial in parentheses.  For example,
	$E_{11} \bullet X = 0$ because $f-\lambda$ has no $x^4$ term. Note that 
	in \eqref{eqn-E11bulletQ} we indicated 	the equation corresponding to $x^2 y$ and  
		several other equations only by vertical dots.

	Of course, we know  that $f-\lambda$ is not SOS for any $\lambda, \, $  hence \eqref{eqn-E11bulletQ} is infeasible.
	We next verify that it is weakly infeasible, and  in the echelon form 
		\eqref{equation-p-weak} without ever having to reformulate it. 
		
		We see that $A_1 := E_{11}, A_2 := E_{22}, \dots, A_5 := E_{55} +  E_{12} +  E_{36} +  E_{47}$ is in semidefinite echelon form, 
		hence the equations in \eqref{eqn-E11bulletQ} 
	prove it  is infeasible 
	(the last right hand side is $-3,$ not  $-1$ as  in Theorem \ref{thm-weak-ref}, but this does not matter). 
	
	On the other hand, let 
	\begin{equation}
	\begin{array}{rcl} 
	X_1 & := & E_{88} \\
	X_2 & := & 2 E_{33} + 2 E_{44} - E_{18} - E_{28} \\
	X_3 & := & E_{55} + E_{66} + E_{77} - E_{47} - E_{36}. 
	\end{array} 
	\end{equation}
Then $(X_1, X_2, X_3)$ is in semidefinite echelon form and proves that \eqref{eqn-E11bulletQ} is not strongly infeasible. To see why, we write 
	the equations in \eqref{eqn-E11bulletQ} as $\A X = b, \, $ then we can check that 
	$\A X_1 = \A X_2 = 0$ and $\A X_3 = b.$ 	
	
	In Figure \ref{figure-motzkin} we visualize the certificates  of infeasibility (on the top)  and the certificates of not-strong-infeasibility (on the bottom). 
		
\end{example}

To better explain Example \ref{example-motzkin}, we make three remarks.

First, SDPs that come  from polynomial optimization problems are widely known to be difficult,  both due to their often pathological behavior, and also due to their size. On the one hand, some remedies to address the difficult behaviors are available. For example,  if such an SDP is feasible, 
we can ensure strong duality by adding a redundant ball constraint, 
see Henrion and Josz  \cite{josz2016strong}.  See also references \cite{henrion2005detecting, lasserre2019sdp} mentioned in the introduction.
One may also entirely do away  with the SDP based approach, and either optimize directly over SOS polynomials, see Papp and Yildiz \cite{papp2019sum}; or use a second order conic programming, or linear programming 
relaxation, which is a bit weaker, but much more scalable, see Ahmadi and Majumdar \cite{ahmadi2019dsos}. 
Example \ref{example-motzkin} complements these works: it gives a combinatorial insight into 
why some of the pathologies arise in the first place. 

Second, Waki in \cite{Waki:12}  constructed a library of  weakly infeasible SDPs from the SOS relaxation of polynomial 
optimization problems;  on the other hand \cite{Waki:12} did not provide certificates 
of the kind we study in this work.

Third, suppose we just wish to decide whether $f-\lambda$ is SOS for some {\em fixed} $\lambda.$ 
 For that, we set up an SDP feasibility problem with the constraints of 
 \eqref{eqn-E11bulletQ}, and add the constraint $E_{88} \bullet X = 1 - \lambda.$ Interestingly, this SDP turns out to be  strongly infeasible, as it was proved by Henrion \cite{henrion2011semidefinite}.

We now move on, and in Corollary \ref{corollary-not-closed}  characterize  the underlying operators 
in weakly infeasible SDPs. The discussion in the introduction shows that 
these are  linear operators that map $\psdn$ 
to a nonclosed set. 
These operators were recently 
baptized as  ``bad projections of the psd cone," 
and explored through the lens of algebraic geometry   \cite{jiang2020bad}.

\begin{corollary} \label{corollary-not-closed} 
	Suppose $\A: \symn \rightarrow \rad{m}$ is a linear map. Then $\A \psd{n}$ is not closed if and only if
	$\A$ has a reformulation $\A'$ with the following properties: 
	\begin{enumerate}
		\item  \label{corollary-not-closed-1}  $(A_1^\prime, \dots, A_{k+1}^\prime)$ is in semidefinite echelon form,  where $k \geq 1;$ 
		\item  \label{corollary-not-closed-2}  There is  $(X_1, \dots, X_{\ell+1}), \, $ in semidefinite echelon form, where $\ell \geq 1$ 
		 and the matrix of inner products of the  $A_i^\prime$ and $X_j$ matrices looks like
		 \begin{center}
		 	\begin{tikzpicture}[scale=1.0, every node/.style={scale=1.0}]

    
    \begingroup
    \renewcommand*{\arraystretch}{1.2}
    \setlength\arraycolsep{4.7pt}
    
    \node[] at (0.22,0) 
    {
    $
    (A_i^\prime \bullet X_j)_{i=1, j=1}^{m, \ell+1} \, = \, \left( \begin{array}{ccc|c} 0 & \dots & 0    & 0 \\
 0 &    \ddots  & 0 & 0 \\
 0 & \dots       & 0 & 0  \\ \hline 
 0 & \dots       & 0 & -1   \\  \hline 
 0  & \dots      & 0 & \times   \\
      & \ddots   &     &        \\
 0   & \dots     & 0 & \times  \end{array}  \right) 
    $
    };
    
    \endgroup
    
    \draw[decoration={brace,raise=4.0pt},decorate] (0.35,2) -- node[above=7.5pt] {$\ell+1$} (2.75,2);
    
    \draw[decoration={brace,raise=4.0pt},decorate] (3.1,1.85) -- node[right=7.5pt] {$k+1$} (3.1,-0.25);
    
    \end{tikzpicture}
		 \end{center}		
	where the $\times$ symbols denote arbitrary elements.
		\end{enumerate}
			\qed 
\end{corollary}

\begin{proof}

To show the ``only if'' direction, suppose $\A \psd{n}$ is not closed, and suppose 
$b \in \rad{m}$ is in the closure of $\A \psd{n}$ but   $b \not \in \A \psd{n}.$ 
Then \eqref{p} is weakly infeasible, so we appeal to Theorem \ref{thm-weak-ref} and construct 
$\A^\prime, b^\prime, \,$ and $X_1, \dots, X_{\ell+1}$ therein.  Then the matrix $(A_i^\prime \bullet X_j)_{i=1, j = 1}^{m, \ell+1}$  
is in the form given in Figure \ref{figure-Ai-Xj-matrix}, so 
items \ref{corollary-not-closed-1} and  \ref{corollary-not-closed-2} in our corollary hold.  

For the ``if'' direction, suppose $\A^\prime$ and $(X_1, \dots, X_{\ell+1})$ are as in the statement of Corollary 
\ref{corollary-not-closed}, and let $b^\prime = \A^\prime X_{\ell+1}.$ Then $(\A^\prime, b^\prime)$ and  
$(X_1, \dots, X_{\ell+1})$  satisfy items \ref{thm-weak-1} and \ref{thm-weak-2} in Theorem \ref{thm-weak-ref}. 
Hence the system \eqref{equation-p-weak} therein is weakly infeasible, so $b^\prime$ is in the closure of 
$\A^\prime \psd{n}$ but $b^\prime \not \in \A^\prime \psd{n}.$  
Thus  $\A^\prime \psd{n}$ is not closed, hence neither is $\A \psd{n}, $ as required.  
\end{proof}

We next contrast Corollary \ref{corollary-not-closed} with an equivalent  characterization of nonclosedness 
of 	$\A \psd{n}$ that we recap below in Theorem \ref{thm-not-closed}. 
Theorem  \ref{thm-not-closed} is obtained from 
\cite[Theorem 1]{pataki2019characterizing} by setting $B=0.$

\begin{theorem} \label{thm-not-closed} 
	Suppose that $Z$ is a maximum rank psd matrix in $\R(\A^*), \, $ the linear span of $A_1, \dots, A_m.$ 
	Assume without loss of generality that $Z$ is of the form 
	\begin{equation} \label{eqn-Z} 
	Z \, = \, \begin{pmatrix} I_r & 0 \\
	                       0  & 0 \end{pmatrix}
	\end{equation}
	for some $r \in \{0, \dots, n \}.$ 
	Then $\A \psd{n}$ is not closed if and only if there is a matrix $V \in \R(\A^*)$ of the form 
	\begin{equation} \label{eqn-V}  
	V \, = \, \begin{pmatrix} V_{11} & V_{12} \\
	                       V_{12}^{\top} & V_{22} \end{pmatrix},
	\end{equation}
	where $V_{22} \in \psd{n-r} $ and $\R(V_{12}^{\top}) \not \subseteq \R(V_{22}).$ 
	\qed
\end{theorem}

Next we  argue that Corollary \ref{corollary-not-closed} is  more useful than 
Theorem \ref{thm-not-closed}, although the latter is more compact. 
In particular, Corollary \ref{corollary-not-closed} can be used to construct 
maps under which $\psd{n}$ is not closed, as we show in Section \ref{section-generate}. 

On the other hand, 
Theorem \ref{thm-not-closed}  cannot be used for this purpose in a straightforward manner. Suppose indeed that we try to construct such an $\A$ and 
choose $A_1 := Z$ as in \eqref{eqn-Z} (with $0 < r < n$) and 
$A_2 := V$ as in \eqref{eqn-V} as elements of $\R(\A^*). \,$ However, we cannot guarantee that 
$Z$ remains a maximum rank psd matrix in $\R(\A^*)$ after we chose the other $A_i.$ 

\section{How to construct  any weakly infeasible SDP and bad projection of the psd cone}
\label{section-generate} 

We now build on Theorem \ref{thm-weak-ref}, and present a combinatorial algorithm, Algorithm 
\ref{weaksdpalgo},   to construct any weakly infeasible SDP of the form \eqref{p}.  Algorithm 
\ref{weaksdpalgo} (with Algorithm \ref{base-eqn-algo} as a subroutine) also constructs any bad projection of the psd cone, 
and provides a vector $b$ in the closure of $\A \psdn$  such that 
$b \not \in \A \psdn.$ 

Algorithm \ref{weaksdpalgo} first chooses positive integers $m, \, k$ and $\ell,$ such that $k+1 \leq m,  \, $ 
then  constructs a weakly infeasible SDP with just $k+1$ constraints in the echelon form 
of   \eqref{equation-p-weak}.  It also 
chooses  a sequence 
$(X_1, \dots, X_{\ell+1})$ in semidefinite echelon form  that certifies not strong infeasibility of this ``small" SDP.  Finally,  it chooses the remaining $m-(k+1)$ 
equality constraints and reformulates the SDP.

Algorithms \ref{weaksdpalgo} and \ref{base-eqn-algo} rely on Theorem \ref{thm-weak-ref}, but to simplify notation, everywhere 
we  write $A_i$ in place of $A_i'$ and $b_i$ in place of $b_i'.$

\begin{algorithm}[H]
	\caption{Construct Weakly Infeasible  SDP} 
	\label{weaksdpalgo}
	\begin{algorithmic}[1]
		\State \label{alg-1} Choose $m, k$ and $\ell$ positive integers such that $k+1 \leq m.$ Also choose $(A_1, \dots, A_{k+1})$ and $(X_1, \dots, X_{\ell+1})$ in semidefinite echelon form,  which satisfy the following {\em base equations:} 
				\begin{equation} \label{eqn-base} \tag{$\mathrm{BASE}$} 
			\boxed{		
				A_i \bullet X_j = \begin{cases}
				0 &\text{if $(i,j)\ne (k+1,\ell+1)$}\\
				-1 &\text{if $(i,j)= (k+1,\ell+1)$}\\
				\end{cases} 
			}
			\end{equation} 
					\State \label{alg-1.5} Let $(b_1, \dots, b_k, b_{k+1}) = (0, \dots, 0, -1).$ 
					\State  \label{alg-2} Choose $A_{k+2},\dots,A_m$ so they have zero $\bullet$ product 
					with $X_1,\dots,X_{\ell}.$
					\State \label{alg-3} Set $b_i:=A_i\bullet X_{\ell+1}$ for $i=k+2,\dots,m.$
					\State \label{alg-4} Reformulate \eqref{p}. 
						\end{algorithmic}
\end{algorithm}

Observe that step \ref{alg-1} ensures that the first $k+1$ rows of the 
$(A_i \bullet X_j)_{i,j=1}^{m, \ell+1}$ matrix are as  required 
in Figure \ref{figure-Ai-Xj-matrix}. Steps \ref{alg-2} and \ref{alg-3} ensure that the rest of the matrix looks like as required in the same figure.  

The only nontrivial step in Algorithm \ref{weaksdpalgo} is step \ref{alg-1}, so  the question is, how to carry out this step?

The main idea is that we have many  potentially 
nonzero blocks in the $A_i$ and the $X_j,$ and  only a 
small number of equations to satisfy. 
Precisely, by a straightforward count the $A_i$ altogether have at  least constant times $k^3$ potentially nonzero blocks of the form 
	$A_i(P_s, P_t),$  where 
 $\{ P_1, \dots, P_{k+1} \}$ is the structure of  $(A_1, \dots, A_{k+1}).$  Similarly, the $X_j$ have at  least constant times $\ell^3$ potentially nonzero blocks. 
	So if we set these blocks in the right order, then the $A_i$ and $X_j$ will satisfy the \eqref{eqn-base}  equations, 	of which there are only $(k+1)(\ell+1).$

 To carry out this plan, we need two lemmas. 
\begin{lemma}\label{pqempty} 
			Suppose that 
			$(A_1, \dots, A_{k+1})$  is in semidefinite echelon form 
		    with structure $\{ P_1, \dots, P_{k+1}\}$ and 
			$(X_1, \dots, X_{\ell+1})$ is   in semidefinite echelon form
			 with structure $\{Q_1, \dots, Q_{\ell+1}\}.$ 
			
			Also suppose that 	$(A_1, \dots, A_{k+1})$  and  
			$(X_1, \dots, X_{\ell+1})$ satisfy the base equations \eqref{eqn-base}.
			 Then 
	\begin{eqnarray} \label{eqn-P1-Qj} 
	P_1 \cap  ( Q_1 \cup \dots \cup Q_{\ell+1}) & = & \emptyset \\ \label{eqn-Q1-Pj}  
	Q_1 \cap (P_1 \cup \dots \cup P_{k+1}) & = & \emptyset. 
	\end{eqnarray} 
\end{lemma}
\begin{proof} We prove \eqref{eqn-Q1-Pj}, the  proof of  \eqref{eqn-P1-Qj} is analogous. 
Since $(A_1, \dots, A_{k+1})$ is in semidefinite echelon form, and $X_1 \succeq 0, \, $ 
an argument like in the ``if" direction in Theorem \ref{thm-weak-ref} proves $X_1(P_1 \cup \dots \cup P_{k+1}, N) = 0.$ 
Since the only nonzero entries of  $X_1$ are in $X_1(Q_1), \, $ the statement follows.
	
\end{proof}

Lemma \ref{bilinlemma} shows how to solve a linear system of equations in an unusual setup, in which only the right hand side is fixed.  In Lemma \ref{bilinlemma} we index the $b_j$ reals  and $Y_j$ matrices from $2$ to $\ell+1$ for convenience. We also define the inner product of possibly nonsymmetric matrices $M$ and $Y$ as the  trace of $M^\top Y.$ 
\begin{lemma}\label{bilinlemma}
	Given positive integers $p, q$ and $\ell, \,$ and real numbers $b_2, \dots, b_{\ell+1}$  there is a polynomial time algorithm to find 
	$M, Y_2, \dots, Y_{\ell+1}$ in $\rad{p \times q}$ such that
	\begin{align}
	\begin{split}\label{bilinlemmasys}
	M \bullet Y_2 &= b_2 \\
	\vdots &\\
	M \bullet Y_{\ell+1}   &= b_{\ell+1}.  
	\end{split}
	\end{align}
	Further, any solution to \eref{bilinlemmasys} is a possible outcome of this algorithm.
\end{lemma}
\begin{proof}
If all $b_j$ are zero, we first choose an arbitrary $M, \, $ then choose 
$Y_2, \dots, Y_{\ell+1}$ to solve the system \eqref{bilinlemmasys}. 
If not all $b_j$ are zero, we  do the same, but we make sure to pick $M \neq 0.$ 

\end{proof}

Algorithm \ref{base-eqn-algo}, which is used as a subroutine in Algorithm
\ref{weaksdpalgo}, constructs sequences of $A_i$ and $X_j$ that satisfy the  \eqref{eqn-base} equations.
\begin{algorithm}[H]
	\caption{Base Equations Algorithm} 
	\label{base-eqn-algo} 
	\begin{algorithmic}[1]
		\State Choose 
	$(A_i)_{i=1}^{k+1}$ and  $(X_j)_{j=1}^{\ell+1}$  in semidefinite echelon form with structure $\{P_i\}_{i=1}^{k+1}$ and  $\{Q_j\}_{j=1}^{\ell+1}$ respectively,  which satisfy $$P_1 \neq \emptyset, \dots, P_k \neq \emptyset, \, Q_1 \neq \emptyset,  \dots, Q_\ell \neq \emptyset, \, \, \eqref{eqn-P1-Qj},  \,\, \text{and} \,\,  \eqref{eqn-Q1-Pj}. $$  
		 	\For{$i = 2:k+1$}
		\State \label{step} Set $A_i(P_{i-1},Q_1),X_2(P_{i-1},Q_1),\dots,X_{\ell+1}(P_{i-1},Q_1)$ to satisfy the base equations \eqref{eqn-base} with left hand side $A_i \bullet X_2, \dots, A_i \bullet X_{\ell+1}.$ 
		\EndFor
	\end{algorithmic}
\end{algorithm}

\begin{lemma} \label{prop-baseeqn} 
	We can implement step \ref{step} of Algorithm \ref{base-eqn-algo} so  the algorithm is correct. 
\end{lemma}
\begin{proof}
For brevity, we define $P_{k+2} := N \setminus (P_1 \cup \dots \cup P_{k+1})$ and  
we fix $i \in \{2, \dots, k+1\}.$ 
	To implement step  \ref{step}   
	of Algorithm \ref{base-eqn-algo} we first set the $(P_{i-1},Q_1)$ block of $A_i,X_2,\dots,X_{\ell+1}$ to zero,  then  introduce a target vector $(b_2, \dots, b_{\ell+1})^\top:$ 
	$$
	b_{j}=\begin{cases}
- \frac{1}{2}  	A_i \bullet X_j &\text{if $(i,j)\ne (k+1,\ell+1)$}\\
- \frac{1}{2} \bigl( A_i \bullet X_j + 1 \bigr) &\text{if $(i,j)= (k+1,\ell+1)$}. \\
	\end{cases}
	$$
	Next we invoke  Lemma \ref{bilinlemma} with $A_i(P_{i-1},Q_1)$ in place of $M$ and 
	$X_j(P_{i-1},Q_1)$ in place of $Y_j$ for $j=2, \dots, \ell+1.$ Finally we symmetrize $A_i$ and the $X_j,$ namely we set 
	 $$ A_i(Q_1, P_{i-1}) := A_i(P_{i-1},Q_1)^{\top} \text{and } X_j(Q_1, P_{i-1}) :=    X_j(P_{i-1}, Q_1)^{\top}$$ {for} $j=2, \dots, \ell+1.$ 
	
Suppose we perform 
	step \ref{step} with a certain $i \in \{2, \dots, k+1 \}$ to satisfy 
	the base equations \eqref{eqn-base} with left hand side
	$A_i \bullet X_j$ for $j=2, \dots, \ell+1.$ 
	We next show that the previously satisfied equations remain true. 
	
If $i=2$ then there is nothing to show, so assume $i \geq 3.$ 
	Let us fix $t \in \{2, \dots, i-1 \}$ and 
	$j \in \{2, \dots, \ell+1 \}.$ We will show that the equation 
	with left hand side $A_{t} \bullet X_j$ remains satisfied. 
		For that, we note that the support of $A_t$ is contained in the blocks 
	\begin{equation} \label{eqn-blocks} 
	\begin{array}{rcl}
	\I_1 & := & (P_1 \cup \dots \cup P_{t}, P_1 \cup \dots \cup P_{t}) \\
	\I_2 & := & (P_1 \cup \dots \cup P_{t-1}, P_{t+1} \cup \dots \cup P_{k+2}) \\
	\I_3 & := & (P_{t+1} \cup \dots \cup P_{k+2}, P_1 \cup \dots \cup P_{t-1}) 
	\end{array}
	\end{equation}
	cf. Definition \ref{definition-reform}. 
	In iteration $i$ we change $X_j(P_{i-1}, Q_1)$ and $X_j(Q_1, P_{i-1})$ for $j=2, \dots, \ell+1.$ So it suffices to show 
	\begin{eqnarray} \label{empty-1} 
	(P_{i-1}, Q_1) \cap \I_1 & = & \emptyset \\ \label{empty-2} 
	(P_{i-1}, Q_1) \cap \I_2 & = & \emptyset \\ \label{empty-3} 
	(P_{i-1}, Q_1) \cap \I_3 & = & \emptyset.
	\end{eqnarray}
	Indeed, \eqref{empty-1} and 	\eqref{empty-3} follow by \eqref{eqn-Q1-Pj}. Further, \eqref{empty-2} follows from $t-1 < i-1$ and the proof is complete.

\end{proof}
	
In summary, we have  the following theorem.
\begin{theorem} \label{theorem-constructall} 
	Algorithm \ref{weaksdpalgo}  always correctly constructs 
	 a weakly infeasible SDP
	 and any weakly infeasible SDP of the form 	\eqref{p} is among its  outputs.
\end{theorem}
\begin{proof}
Lemmas \ref{pqempty}--\ref{prop-baseeqn} imply that 	Algorithm \ref{weaksdpalgo}   always correctly outputs a weakly infeasible SDP. On the other hand, suppose  \eqref{p} is weakly infeasible.
Then \eqref{p} has a reformulation \eqref{equation-p-weak} as presented in Theorem  \ref{thm-weak-ref}.
For simplicity, let us denote the  operator in   \eqref{equation-p-weak} by  $\A$ and represent  $\A$ with matrices 
$A_1, \dots, A_m.$  Also let us denote the right hand side in   \eqref{equation-p-weak} by    $b.$ 

Assume the first $k +1 \geq 2$ equations in   \eqref{equation-p-weak} prove it is infeasible. 
We know that \eqref{equation-p-weak} is not strongly infeasible, and we 
 let $(X_1, \dots, X_{\ell+1})$ be  the sequence 
that certifies this as presented in Theorem \ref{thm-weak-ref}. Recall that $\ell \geq 1.$

Suppose that  $(A_1, \dots, A_{k+1})$ has structure $\{P_1, \dots, P_{k+1}\}$ and 
$(X_1, \dots, X_{\ell+1})$ has structure 
$\{Q_1, \dots, Q_{\ell+1}\}.$ By the remark following Theorem \ref{thm-weak-ref} we can assume that $P_1, \dots, P_k$ and 
$Q_1, \dots, Q_\ell \,$ are nonempty.
Hence by Lemma \ref{bilinlemma} we have that  $(A_1, \dots, A_{k+1})$ and 
$(X_1, \dots, X_{\ell+1})$ are possible outputs of Algorithm \ref{base-eqn-algo}. 
Besides, 
$A_{k+2}, \dots, A_m$ and $b_{k+2}, \dots, b_m$ 
are possible outputs of steps \ref{alg-2} and \ref{alg-3} in Algorithm \ref{weaksdpalgo}.
Because of the reformulation step  \ref{alg-4} we see that \eqref{p} is a possible output of 
Algorithm \ref{weaksdpalgo}, and the proof is complete.

\end{proof}

As a quick check, Algorithm \ref{weaksdpalgo}  constructs a variant of \eqref{problem-small} as follows.
First it sets $m=2, \, k=\ell=1, \, P_1 = \{1\}, \, Q_1 = \{2\}, \, P_2 = Q_2 = \emptyset$ and 
$$
A_1 \, = \, \begin{pmatrix}  1 & 0 \\ 0 & 0 \end{pmatrix}, \, A_2 \, = \, \begin{pmatrix} 0 & \alpha \\ \alpha &  0  \end{pmatrix},  \, X_1 = \begin{pmatrix} 0 & 0 \\ 0 & 1 \end{pmatrix},  \, X_2 = \begin{pmatrix} 0 & \beta \\ \beta & 0 \end{pmatrix}, 
$$ 
where $\alpha$ and $\beta$ are arbitrary. Then the subroutine Algorithm \ref{base-eqn-algo} sets $\alpha$ and $\beta$ to satisfy $2 \alpha \beta = -1.$ 
 Algorithm \ref{weaksdpalgo} in step \ref{alg-1.5} sets  $b = (0, -1)^\top. $ Then it  skips 
 steps \ref{alg-2} and \ref{alg-3} (since $m = k+1$) and also skips the reformulation of step \ref{alg-4}.

\begin{example} \label{example-3by3} 
	Let $k= \ell = 1, \, P_1 = \{1\}, \, P_2 = \{2\}, \, Q_1 = \{3\}, Q_2 = \{2\}, \, \alpha$ and $\beta$ be arbitrary reals, and 
	$$
	A_1 = \begin{pmatrix} 1 & 0 & 0 \\
	                    0 & 0 & 0 \\
	                    0 & 0 & 0 \end{pmatrix}, \, A_2 =  \begin{pmatrix} 0 & 0 & \alpha \\
	                                                                         0 & 1 & 0 \\
	                                                                         \alpha & 0 & 0 \end{pmatrix}, \,  X_1 = \begin{pmatrix} 0 & 0 & 0 \\
	                                                                         0 & 0 & 0 \\
	                                                                         0 & 0 & 1 \end{pmatrix}, \, X_2 =  \begin{pmatrix} 0 & 0 & \beta \\
	                                                                          0 & 1 & 0 \\
	                                                                         \beta & 0 & 0 \end{pmatrix}.
	$$
Our algorithms  construct a weakly infeasible SDP from this data 
as follows.   Algorithm \ref{base-eqn-algo} sets $\alpha$ and $\beta$ so that 
$\alpha \beta = -1.$ After this the $A_i$ and $X_j$ satisfy the 	\eqref{eqn-base}  equations.  
Then Algorithm \ref{weaksdpalgo}  in step \ref{alg-1.5} sets $(b_1, b_2) = (0,-1)$ and in step 
\ref{alg-2} it chooses 
	$$
	A_3 : =  \begin{pmatrix} 0 & 1 & 0 \\
	1 & 0 & 1 \\
	0 & 1 & 0 \end{pmatrix},  
	$$
	 as it has zero inner product with $X_1.$ 
	Finally, in step \ref{alg-3} it sets $b_3 := A_3 \bullet X_2 =   0$  and skips the reformulation of step \ref{alg-4}. 
	\end{example}
We note that a scheme to construct weakly infeasible SDPs was given in \cite{liu2017exact}. That scheme is somewhat similar to the one presented here, as it uses a sequence of $A_i$ matrices to certify infeasibility and a  sequence of $X_j$ matrices to certify not strong infeasibility.  However, the  scheme in \cite{liu2017exact} assumes that the positive definite blocks of the $A_i$ and $X_j$ matrices do not overlap, so the variety of weakly infeasible SDPs it can construct is limited. In particular, it cannot even construct the small SDP of Example \ref{example-3by3}. 

On the other hand, our Algorithm \ref{weaksdpalgo}, that generates {\em any}  weakly infeasible SDP,  relies on Theorem 
\ref{thm-weak-ref}; that result is simple to state, but somewhat technical to prove. 

\begin{example} (Example \ref{example-large} continued) 
We now show how Algorithm \ref{weaksdpalgo} constructs the SDP in Example \ref{example-large}. 	
We first set $ m=4, k = \ell = 2, \,$ and select  $(A_1,A_2,A_3)$ and $(X_1,X_2,X_3)$ in semidefinite echelon form shown below: 
\begin{figure}[H]
	\begin{center} 
		\begin{tikzpicture}[scale=0.5, every node/.style={scale=0.6}]
    
    %
    \newcommand{\tixmat}[3]{
    \foreach \x in {2,...,#3} \draw[] (\x-1+#1,#2) to (\x-1+#1,#2+#3);
    \foreach \y in {2,...,#3} \draw[] (#1,\y-1+#2) to (#1+#3,\y-1+#2);
    \draw[] (#1,#2) rectangle (#1+#3,#2+#3);
    }
    
    \fill[red!60] (0,3) rectangle (1,4);
    
    \fill[cyan!80] (5+0,0) rectangle (5+1,4);
    \fill[cyan!80] (5+1,3) rectangle (5+4,4);
    \fill[red!60] (5+1,2) rectangle (5+2,3);
    
    \fill[cyan!80] (10+0,0) rectangle (10+2,4);
    \fill[cyan!80] (10+2,2) rectangle (10+4,4);
    \fill[red!60] (10+2,1) rectangle (10+3,2);
    
    \fill[red!60] (3,-6) rectangle (4,-6+1);
    
    \fill[cyan!80] (5+0,-6+0) rectangle (5+4,-6+1);
    \fill[cyan!80] (5+3,-6+1) rectangle (5+4,-6+4);
    \fill[red!60] (5+1,-6+2) rectangle (5+2,-6+3);
    
    \fill[cyan!80] (5*2+0,-6+0) rectangle (5*2+4,-6+1);
    \fill[cyan!80] (5*2+0,-6+2) rectangle (5*2+4,-6+3);
    \fill[cyan!80] (5*2+1,-6+0) rectangle (5*2+2,-6+4);
    \fill[cyan!80] (5*2+3,-6+0) rectangle (5*2+4,-6+4);
    \fill[red!60] (5*2+2,-6+1) rectangle (5*2+3,-6+2);
    
    \tixmat{0}{0}{4};
    \tixmat{5}{0}{4};
    \tixmat{10}{0}{4};
    
    \tixmat{0}{-6}{4};
    \tixmat{5}{-6}{4};
    \tixmat{10}{-6}{4};
    
    \node[] at (2+5*2,-0.75) {\LARGE $A_3$};
    \node[] at (2+5,-0.75) {\LARGE $A_2$};
    \node[] at (2,-0.75) {\LARGE $A_1$};
    
    \node[] at (2+5*2,-0.75-6) {\LARGE $X_3$};
    \node[] at (2+5,-0.75-6) {\LARGE $X_2$};
    \node[] at (2,-0.75-6) {\LARGE $X_1$};
    
    \draw[decoration={brace,raise=1.25pt},decorate] (0,3.05) -- node[left=6.75pt] {\Large $P_1$} (0,3.95);
    \draw[decoration={brace,raise=1.25pt},decorate] (0,2.05) -- node[left=6.75pt] {\Large $P_2$} (0,2.95);
    \draw[decoration={brace,raise=1.25pt},decorate] (0,1.05) -- node[left=6.75pt] {\Large $P_3$} (0,1.95);
    
    \draw[decoration={brace,raise=1.25pt},decorate] (0,0.05-6) -- node[left=6.75pt] {\Large $Q_1$} (0,0.95-6);
    \draw[decoration={brace,raise=1.25pt},decorate] (0,2.05-6) -- node[left=6.75pt] {\Large $Q_2$} (0,2.95-6);
    \draw[decoration={brace,raise=1.25pt},decorate] (0,1.05-6) -- node[left=6.75pt] {\Large $Q_3$} (0,1.95-6);
    
    \node[] at (0.5,3.5) {\LARGE $1$};
    
    \node[] at (5+2.5,3.5) {\LARGE $-2$};
    \node[] at (5+3.5,3.5) {\LARGE $*$};
    \node[] at (5+0.5,1.5) {\LARGE $-2$};
    \node[] at (5+0.5,0.5) {\LARGE $*$};
    \node[] at (5+1.5,2.5) {\LARGE $1$};
    \node[] at (5+0.5,2.5) {\LARGE $0$};
    \node[] at (5+0.5,3.5) {\LARGE $2$};
    \node[] at (5+1.5,3.5) {\LARGE $0$};
    
    \node[] at (10+3.5,2.5) {\LARGE $*$};
    \node[] at (10+1.5,0.5) {\LARGE $*$};
    \node[] at (10+2.5,1.5) {\LARGE $1$};
    \node[] at (10+0.5,0.5) {\Large $3/2$};
    \node[] at (10+0.5,1.5) {\LARGE $1$};
    \node[] at (10+0.5,2.5) {\LARGE $3$};
    \node[] at (10+0.5,3.5) {\LARGE $0$};
    \node[] at (10+1.5,3.5) {\LARGE $3$};
    \node[] at (10+2.5,3.5) {\LARGE $1$};
    \node[] at (10+3.5,3.5) {\Large $3/2$};
    \node[] at (10+2.5,2.5) {\LARGE $4$};
    \node[] at (10+1.5,1.5) {\LARGE $4$};
    \node[] at (10+1.5,2.5) {\LARGE $1$};
    
    \node[] at (3.5,0.5-6) {\LARGE $1$};
    
    \node[] at (1+5.5,2+0.5-6) {\LARGE $1$};
    \node[] at (0+5.5,0+0.5-6) {\LARGE $*$};
    \node[] at (1+5.5,0+0.5-6) {\LARGE $*$};
    \node[] at (2+5.5,0+0.5-6) {\LARGE $0$};
    \node[] at (3+5.5,0+0.5-6) {\LARGE $-4$};
    \node[] at (3+5.5,1+0.5-6) {\LARGE $0$};
    \node[] at (3+5.5,2+0.5-6) {\LARGE $*$};
    \node[] at (3+5.5,3+0.5-6) {\LARGE $*$};
    
    \node[] at (2+10.5,1+0.5-6) {\LARGE $1$};
    \node[] at (0+10.5,0+0.5-6) {\LARGE $*$};
    \node[] at (1+10.5,0+0.5-6) {\LARGE $*$};
    \node[] at (2+10.5,0+0.5-6) {\LARGE $-4$};
    \node[] at (3+10.5,0+0.5-6) {\LARGE $2$};
    \node[] at (3+10.5,1+0.5-6) {\LARGE $-4$};
    \node[] at (3+10.5,2+0.5-6) {\LARGE $*$};
    \node[] at (3+10.5,3+0.5-6) {\LARGE $*$};
    \node[] at (2+10.5,2+0.5-6) {\LARGE $-2$};
    \node[] at (1+10.5,2+0.5-6) {\LARGE $0$};
    \node[] at (0+10.5,2+0.5-6) {\LARGE $4$};
    \node[] at (1+10.5,3+0.5-6) {\LARGE $4$};
    \node[] at (1+10.5,1+0.5-6) {\LARGE $-2$};
    
    \end{tikzpicture}
	\end{center}
\end{figure}
\vspace{-0.5cm}
Here at the start the entries marked by  $*$ are arbitrary. 
Algorithm   \ref{base-eqn-algo} first ensures $A_2 \bullet X_2 = A_2 \bullet X_3 = 0 $ by
setting $$ A_2(1,4)=1/2, \quad X_2(1,4)=-1, \quad X_3(1,4)=0.$$ 
Then it ensures  
$A_3 \bullet X_2=0,A_3 \bullet X_3=-1$ by setting 
$$ A_3(2,4)=1, \quad X_2(2,4) = 1, \quad X_3(2,4)=-5.$$ (It also sets the symmetric blocks, 
for example, it sets $A_2(4,1)=1/2$ and so on.) 

Next, Algorithm   \ref{weaksdpalgo} in step \ref{alg-1.5} 
sets $(b_1, b_2, b_3) = (0,0,-1),$ then  in step \ref{alg-2} it selects 
$$
A_4 := \frac{1}{2}\begin{pmatrix}
-1 & -2 & -1 & 3 \\
-2 & 2 & -1 & 2 \\
-1 & -1 & 0 & -1 \\
3 & 2 & -1 & 0
\end{pmatrix}
$$
as it is  orthogonal to $X_1$ and $X_2.$ Then in step \ref{alg-3} it  sets  $b_4 = A_4 \bullet X_3 = -12.$ 
We thus get the 
weakly infeasible instance of Example \ref{example-large}
with the $A_i$ and $X_j$ shown above and 
$b = (0,0, -1, -12)^\top.$ 

\end{example} 

\section{Proofs: certificates of infeasibility and not strong infeasibility separately}
\label{section-certificate-separately}

In this section we construct two distinct reformulations of \eqref{p}: one to certify  it is infeasible, and the second to certify it is not strongly infeasible. Lemma \ref{thm-inf} already appeared 
as part (1) of Theorem 5  in \cite{liu2017exact} and 
Lemma \ref{thm-notstrong} as part (2) of Theorem 5 in 
\cite{liu2017exact}. An earlier version of the  latter result appeared in \cite{lourencco2016structural}. Here we give shorter and more elementary proofs. 

We first state a necessary condition for a semidefinite system to be infeasible.
Lemma \ref{lemma-infeas} is a slightly stronger version of Lemma 3.2 in Waki and Muramatsu 
\cite{WakiMura:12}. 

\begin{lemma}\label{lemma-infeas}   
	Suppose $B \in \sym{n}$ and $L \subseteq \symn$ is a subspace such that $(B + L  )\cap\psdn=\emptyset.$
	Then the  following hold:
	\begin{enumerate}
		\item 	\label{lemma-infeas-1} The system \eqref{equation-YBalt} is feasible: 
		\begin{equation}\label{equation-YBalt} 
		\begin{array}{rl}
		B \bullet Y & \leq 0\\
		Y & \in   L^\perp\cap \bigl( \psdn\setminus\{0\} \bigr). 
		\end{array}
		\end{equation}
	\item 	\label{lemma-infeas-2}  	If \eqref{equation-YBalt}  has a positive definite solution, then it has a positive definite solution $Y$ such that $B \bullet Y < 0.$
	\end{enumerate}
\end{lemma}
\begin{proof}  To prove item \ref{lemma-infeas-1}, let $\B : \rad{m} \rightarrow \sym{n}$ be a linear map whose rangespace is $L.$ We claim that the optimal value of the SDP 
\begin{equation} \label{dref1}
\begin{array}{rl}
\sup  &  y_0 \\
s.t.    &  - B y_0 - \B y \preceq 0 
\end{array} 
\end{equation}
is zero. Indeed, its optimal value is nonnegative, since $(y, y_0) = (0, 0)$ is feasible in it. 
On the other hand, the optimal value cannot be positive: if $y_0 > 0$ were feasible with some $y, \, $ then we would get the contradiction $B + \frac{1}{y_0} \B y \succeq 0. \, $  

First assume that \eref{dref1} satisfies Slater's condition. Then its dual (which is of the form 
\eqref{p-opt}), is feasible. Any  $Y$ feasible in the dual of  \eqref{dref1} satisfies 
$$
B \bullet Y  = -1, \quad \B^* Y = 0,
$$
as required.

Second, assume that \eref{dref1} does not  satisfy  Slater's condition.  We claim that the optimal value of the SDP 
\begin{equation} \label{equation-supt} 
\begin{array}{rl}
\sup  &  t  \\
s.t.    &  t  I - B y_0 - \B y \preceq 0 
\end{array} 
\end{equation}
is zero. Indeed,  it is nonnegative since setting all variables to zero we obtain a feasible  solution. On the other hand if $t  > 0$ were feasible with some $y_0  $ and $y, \, $ then the contradiction $B y_0 + \B y \succeq t I \succ 0 \, $ would follow.  

Note that 
\eqref{equation-supt} does satisfy  Slater's condition with $t=-1, \, y = 0, \, $ and $y_0 = 0.$ Thus  there is a $Y$ feasible in the  dual of \eqref{equation-supt}, which satisfies 
$$
B \bullet Y  = 0 , \quad \B^* Y = 0, \quad  I \bullet Y = 1,
$$
as required.

To prove item \ref{lemma-infeas-2}, we observe that $(B + L  )\cap\psdn=\emptyset$ implies  $B \not \in L. \, $ So there is $Y' \in L^\perp$ such that $$B \bullet Y'  < 0.$$ Suppose 
	\eqref{equation-YBalt} has a 
	positive definite solution. Then we  add a sufficiently small positive multiple of $Y'$ to it and 
 obtain a 
	$Y$ positive definite feasible solution such that \mbox{$B \bullet Y < 0.$} The proof is now complete.

\end{proof}

\begin{lemma}\label{thm-inf}
The SDP \eqref{p} is infeasible if and only if it has a reformulation
\begin{equation}\label{equation-p-infeas} \tag{$\mathrm{P}_{\rm infeas}$} 
\begin{array}{rcl} 
\A' X & = & b' \\  
X & \succeq & 0
\end{array}
\end{equation}
in which  $(A_1^{\prime},\dots,A_{k+1}')$ is in semidefinite echelon form 
and 
$(b_1',\dots,b_k',b_{k+1}') = (0, \dots, 0, -1)$   for some $k \geq 0.$ 
\end{lemma} 
   		
\begin{proof}

The ``if" direction is given verbatim in the proof of the ``if" direction in Theorem \ref{thm-weak-ref}. 
 For the ``only if" direction, by elementary row operations (operations \ref{exch} and \ref{trans} in Definition \ref{definition-reform}) we will first  achieve the following:
	\begin{equation} \label{eqn-toachieve1}   
A_1  \,  \succeq \,   0,  \;   b_1 \, \in \{ 0, -1 \}          
\end{equation}
and 
\begin{equation} \label{eqn-toachieve2}   
b_1   \, = \, 0   \;    \Rightarrow  \; 0 \, <  \, \myrank A_1 \, < \, n. 
\end{equation}
For that,  we distinguish two cases.
\begin{case}[\textit{The linear system $\A X =b$ is infeasible}] \label{case1} \normalfont
	Then by elementary linear algebra there is $y$ such that 
	$$\A^* y = 0, \quad b^{\top} y = -1.$$ 
	We have $y \neq 0$ so after permuting the equations in $\A X = b$ we assume $y_1 \neq 0$ without loss of generality.  
	Next, using elementary row operations we replace $A_1$ by $\A^* y$ and $b_1$ by $b^\top y.$ As a result, 
	\eqref{eqn-toachieve1} and  	\eqref{eqn-toachieve2}    hold.
	\end{case}
	\begin{case}[\textit{The linear system $\A X =b$ is feasible}] \normalfont
 We first fix $X_0 \in \symn$ such that $\A X_0=b.$  Then we apply  Lemma \ref{lemma-infeas} with $L := \N(\A)$ and $B := X_0$ and find $Y$ feasible in the system \eqref{equation-YBalt}. Further, using item \ref{lemma-infeas-2} in Lemma \ref{lemma-infeas},  if $Y$
		is positive definite, we  ensure $X_0 \bullet Y < 0.$ 
 We have  $Y \in L^\perp, \, $ and $L^\perp = \R(\A^*), \,$ so we write $Y = \A^*y$ 
 for some $y \in \rad{m}$ and 
deduce 
$$
0 \, \geq \, X_0 \bullet Y \, = \, X_0 \bullet \A^*y \, =  \, (\A X_0)^\top y= b^{\top} y.
$$

We then proceed as in Case \ref{case1}: we permute the equations in $\A X = b,$ if needed, and 
replace $A_1$ by $\A^* y$ and $b_1$ by $b^\top y.$ Afterwards,  
 if $b_1 < 0$ 
then we rescale $A_1$ and $b_1$ so that $b_1 = -1. \, $
Again, \eqref{eqn-toachieve1} and \eqref{eqn-toachieve2} hold. 
\end{case} 

Now that we have satisfied \eqref{eqn-toachieve1} and \eqref{eqn-toachieve2}, we choose an invertible $T$ matrix such that $T^{\top} A_1 T = I_r \oplus 0$ for some $r \geq 0, \, $ 
and let  
\begin{equation}
A_i' := T^{\top} A_i T, \; b_i' := b_i  \;\; \text{for} \; i=1, \dots, m.
\end{equation}
If $b_1' = -1, \, $  we set $k=0$ and stop.  

If $b_1' = 0, \, $ then we must have $0 < r < n.$ We must also have $m > 1, \, $ otherwise 
the all zero matrix would be feasible in \eqref{p}.  We then delete the  
equation $ A_1'\bullet X = 0 \, $ and also delete the first $r$ rows and columns from  the other $A_i'.$ 
We thus obtain a smaller SDP,  say $\mathrm{(P')}, \, $  with $m-1$ equations and order $n-r$ matrices.
We see that  $\mathrm{(P')} \, $  is infeasible: if $X'$ were feasible in it, then $X := 0 \oplus X'$ would be feasible in \eqref{p}.
So we proceed by induction, as  a reformulation of $\mathrm{(P')} \, $ into the form of \eqref{equation-p-infeas} 
yields a reformulation of \eqref{p} into the same form.

\end{proof}

As a quick sanity check, we consider the SDP in  \eqref{problem-small}, and the reformulation given in Example
\ref{example-small-contd} (see equation \eqref{eqn-A1primeA2prime}). This reformulation is 
 in the form \eqref{equation-p-infeas}. Note that now $k=1.$

The proof of Lemma \ref{thm-inf} implies that the positive definite blocks in $A_1',\dots, A_k'$ are nonempty, and can be chosen as identity matrices. However, the positive definite block in $A_{k+1}'$ may be empty, as it is in the reformulated version of 
\eqref{problem-small} given in the previous paragraph.

We next present Lemma \ref{thm-notstrong} to construct a certificate that \eqref{p} is not strongly infeasible.
 Lemma \ref{thm-notstrong} builds on two ideas. First,  if \eqref{p} is not strongly infeasible, then the alternative system \eqref{p-alt} is infeasible. In turn, if \eqref{p-alt}  is infeasible, then using Lemma 
	\ref{thm-inf}  we will reformulate it to make its infeasibility evident. 

\begin{lemma}\label{thm-notstrong} The SDP 
    \eref{p} is not strongly infeasible if and only if     it has a reformulation 
    \begin{equation} \label{equation-p-notstrong}\tag{$\mathrm{P}_{\rm notstrong}$}  
    \begin{array}{rcl} 
    \A'' X & = & b'' \\  
    X & \succeq & 0,
    \end{array}
    \end{equation}
    such that for some   $(X_1,\dots,X_{\ell+1})$ in semidefinite echelon form 
    with $\ell \geq 0$ the following holds:  
    \begin{equation} \label{equation-p-notstrong-2}
    \begin{array}{ccl}
    \mathcal{A}^{\prime\prime}X_i &=& 0 \quad~ \text{for} ~~  i=1,\dots,\ell ~~ \text{and} \\ 
    \mathcal{A}^{\prime\prime}X_{\ell+1} &=& b^{\prime\prime}. 
    \end{array}
    \end{equation}
\end{lemma}
\begin{proof} 

The proof of the ``if" direction is given in the proof of the ``if" direction in Theorem \ref{thm-weak-ref}. 

For the ``only if" direction, we assume that \eqref{p} is not strongly infeasible, and choose an operator $\B:\symn \to \rad{m}$ such that  
$\R(\A^*) = \N(\B).$ We also choose $X_0 \in \symn$ such that $\A X_0=b$ (such an $X_0$ must exist, otherwise 
\eqref{p} would be strongly infeasible).  
Since \eqref{p} is not strongly infeasible, the alternative system \eqref{p-alt}   is infeasible.
We claim that \eqref{p-alt} is equivalent to 
\begin{equation} \label{dalt0:2} 
\begin{array}{rcl}
\B Y  &=& 0 \\
 X_0 \bullet Y &=& -1 \\
Y &\succeq& 0.
\end{array}
\end{equation}
Indeed, by the choice of $\B$ we have that $\B Y=0$ for some $Y \in \symn$ iff 
$Y = \A^* y$ for some $y \in \rad{m}.$ For any such $Y$ and $y$ we see that
$$
b^{\top} y \, = \,  (\A X_0)^\top y  \, = \, X_0 \bullet Y ,
$$
and this proves that  \eqref{p-alt} and \eqref{dalt0:2} are equivalent. 

Thus, by Lemma \ref{thm-inf},  the system  \eqref{dalt0:2} has a reformulation of the form \eqref{equation-p-infeas}, in which for some $\ell \geq 0$ the first $\ell + 1$ equations prove the infeasibility.
These equations are of the  form  
\begin{equation} \label{eqn-Xj} 
\begin{array}{rcl}
X_j \bullet  Y & = & 0 \quad (j=1, \dots, \ell) \\
 X_{\ell+1} \bullet Y & = & -1, 
\end{array} 
\end{equation}
where $(X_1, \dots, X_{\ell+1})$ is in semidefinite echelon form. 

Note that in \eqref{dalt0:2} the only equation with nonzero right hand side is $ X_0 \bullet Y = -1.\, $ 
Given that from \eqref{dalt0:2} we derived equations \eqref{eqn-Xj} by elementary row operations and by  congruence transformations, we see that  
\begin{equation} \label{eqn-Xj-form} 
\begin{array}{rcll} 
X_j  & \in & T^{\top} \R( \mathcal{B}^* ) T & \, \text{for} \, j=1, \dots, \ell  \\
X_{\ell+1} & \in & T^{\top} \bigl( \R( \mathcal{B}^*)  + X_0  \bigr) T 
\end{array} 
\end{equation} 
for some invertible matrix $T.$ 

Observe that $\R( \mathcal{B}^* ) = \N(\A).$   Then from   \eqref{eqn-Xj-form} we deduce that 
for $i=1, \dots, m$  
$$
   A_i \bullet T^{-\top} X_j T^{-1} \, = \,  \left\{    \begin{array}{ll}  0 & \text{if} \quad j \in \{1, \dots, \ell \}  \\
 b_i    &  \text{if} \quad j = \ell + 1   
\end{array}
 \right.
$$
holds. We have $ A_i \bullet T^{-\top} X_j T^{-1} \, = \,  T^{-1} A_i T^{-\top} \bullet X_j$ for all $i.$ 
We define the operator $\A''$  as 
$$
\A''X  = ( A_1'' \bullet X, \dots, A_m'' \bullet X)^{\top},
$$
where $A_i^{\prime \prime}  = T^{-1} A_i T^{-\top}$ for all $i \, $ and let $b'' = b.$ 
We see that $\A'', b'', \, $ and the $X_1, \dots, X_{\ell+1}$ that we already defined satisfy the requirements of our lemma. 

\end{proof}

Yet again, consider the reformulation of  \eqref{problem-small} given in Example \ref{example-small-contd} and set $A_1'' = A_1',$ and $A_2'' = A_2',$ 
 and 
$X_1$ and $X_2$ as in \eqref{eqn-X1X2}. Then $(A_1'', A_2'')$ and $(X_1, X_2)$ with $b''=(0,-1)^{\top}$ satisfy the conclusions of Lemma \ref{thm-notstrong}. Note that now $\ell = 1.$

By the proof of Lemma \ref{thm-notstrong} the positive definite blocks in $X_1,\dots, X_\ell$ are nonempty, and can be chosen as identity matrices. However, the positive definite block in $X_{\ell+1}$ may be empty, as it is in the reformulated version of 
\eqref{problem-small} that we gave in the previous paragraph.

\section{Proof of Theorem \ref{thm-weak-ref}}
\label{section-proof-thm1}

Section \ref{section-certificate-separately} showed how to  produce a reformulation  \eqref{equation-p-infeas} to prove that \eqref{p}
is infeasible; and another reformulation   \eqref{equation-p-notstrong}     to prove  it is not strongly infeasible.
In this section we show that a single reformulation can accomplish both. This common reformulation was fairly straightforward to produce when we started with a simple  problem  like \eqref{problem-small}. 
In the general case we need a technical proof.

We first define operators  that transform a certain targeted block of a matrix.
To absorb Definition \ref{definition-IRG}  we need to recall the notation $M(R,S)$ and $M(R)$ for blocks of a matrix $M$ from the start of Section \ref{section-prelim}.

\begin{definition} \label{definition-IRG} 
	Suppose $R \subseteq N$ and $G$ is matrix of order $|R|$. The matrix $I_{R, G}$ is obtained from the $n \times n$ identity by replacing $I(R)$ by $G,$ i.e., by performing the following two steps: 
	\begin{eqnarray*}
			I_{R,G}  := I, \\ 
		I_{R,G}(R) := G. 
		\end{eqnarray*}
	\end{definition}
For example, if $n=4, \, R = \{1,4 \}, \, $ and $G  = \begin{pmatrix} 2 & 3 \\ 4 & 5 \end{pmatrix}, \, $ then 
$$
I_{R, G} \, = \, \begin{pmatrix} 2 & 0 & 0 & 3 \\
                              0 & 1  & 0 & 0 \\
                               0 & 0 & 1 & 0 \\
                               4 & 0 & 0 & 5 \end{pmatrix}. 
$$
Suppose $M\in\mathbb{R}^{n\times n}$ and $R \subseteq N$. Then the operation 
$$
M :=  M \cdot I_{R,G} 
$$
right multiplies by $G$ the columns of $M$ indexed by $R \,$ and leaves the rest of $M$ unchanged.

Given subsets $R_1, \dots, R_t $ of $N$ and indices $i$ and $j$ such that $1 \leq i \leq j \leq t$ we will use the following shorthand:  
\begin{equation} \label{eqn-RiRj} 
R_{i:j} := R_i \cup R_{i+1} \cup \dots \cup R_j.
\end{equation} 
We will often use a congruence transformation to put matrices into a convenient block diagonal form, bringing us to the following lemma: 

\begin{lemma} \label{lemma-X-rotate} 
	Suppose $X \in \sym{n}$ and $R_1, \dots, R_t$ are disjoint subsets of $N$ such that  
	$$
	X(R_{1:t}) \succeq 0. 
	$$
	Then there is an invertible matrix $T$ such that
	$$
	(T^{\top}XT)(R_{1:t})\text{\rm ~is nonnegative diagonal},
	$$
and  $T$ can be chosen as the product of $n\times n$ invertible matrices 
			$$
			T = I_{R_1, U_1} W_1 \dots I_{R_t, U_t} W_t, 
			$$
			where 
			\begin{enumerate}
				\item the $U_i$ are orthonormal matrices for all $i$.
				\item right multiplying an $n \times n$ matrix, say $M, \, $ 
				by $W_i$  adds multiples of columns in $M(N, R_i)$ to columns of $M(N, R_j)$ for some  $j$ indices in $\{i+1, \dots, t \}$.  
		\end{enumerate}
		\qed
\end{lemma}
Suppose $W_i$ is a matrix in the statement of Lemma \ref{lemma-X-rotate}. We can  describe 
$W_i$  algebraically as follows: i) it has all $1$ entries on the main diagonal;  ii) the block $W_i(R_i, R_j)$ is nonzero for some $j$ indices in $\{i+1, \dots, t \};$ iii) all other blocks of $W_i$ are zero.

\begin{proof}[Proof\;(of Lemma \ref{lemma-X-rotate})]
Let $U_1$ be a matrix of orthonormal eigenvectors of $X(R_1)$ and define $T := 	I_{R_1, U_1}.$ Then $(T^{\top}XT)(R_{1:t})$ looks like on the first picture of Figure \ref{figure:convert}, where the  $\times$ symbols represent  arbitrary elements.

\begin{figure}[H]
	\begin{center}
		\begin{tikzpicture}[scale=0.5, every node/.style={scale=0.6}]

\fill[red!60] (0,5) rectangle (1,6);
\fill[cyan!80] (2,0) rectangle (6,4);
\fill[cyan!80] (0,0) rectangle (1,4);
\fill[cyan!80] (2,5) rectangle (6,6);

\draw (0,2) to (6,2);
\draw (0,4) to (6,4);
\draw (2,0) to (2,6);
\draw (4,0) to (4,6);
\draw (0,0) rectangle (6,6);

\draw (0,5) to (6,5);
\draw (1,6) to (1,0);

\node[] at (0.5,5.5) {\LARGE $+$};
\node[] at (1.5,4.5) {\LARGE $0$};
\node[] at (3,4.5) {\LARGE $0$};
\node[] at (5,4.5) {\LARGE $0$};
\node[] at (1.5,3) {\LARGE $0$};
\node[] at (1.5,1) {\LARGE $0$};
\node[] at (5,1) {\LARGE $\times$};
\node[] at (3,1) {\LARGE $\times$};
\node[] at (3,5.5) {\LARGE $\times$};
\node[] at (5,5.5) {\LARGE $\times$};
\node[] at (5,3) {\LARGE $\times$};
\node[] at (3,3) {\LARGE $\times$};
\node[] at (0.5,3) {\LARGE $\times$};
\node[] at (0.5,1) {\LARGE $\times$};

\draw[decoration={brace,raise=1.25pt},decorate] (0.05,6) -- node[above=6.75pt] {\Large $R_1$} (1.95,6);
\draw[decoration={brace,raise=1.25pt},decorate] (4.05,6) -- node[above=6.75pt] {\Large $R_t$} (5.95,6);
\node[] at (2.5,6.5) {\LARGE $.$};
\node[] at (3,6.5) {\LARGE $.$};
\node[] at (3.5,6.5) {\LARGE $.$};

\draw[decoration={brace,raise=1.25pt,mirror},decorate] (2+0.05,-0.15) -- node[below=6.75pt] {\LARGE $\succeq 0$} (5+0.95,-0.15);

\draw[decoration={brace,raise=1.25pt,mirror},decorate] (11+0.05,-0.15) -- node[below=6.75pt] {\LARGE $\succeq 0$} (14+0.95,-0.15);

\node[] at (7.5,3) {\LARGE $\longmapsto$};

\fill[red!60] (9+0,5) rectangle (9+1,6);
\fill[cyan!80] (9+2,0) rectangle (9+6,4);

\draw (9+0,2) to (9+6,2);
\draw (9+0,4) to (9+6,4);
\draw (9+2,0) to (9+2,6);
\draw (9+4,0) to (9+4,6);
\draw (9+0,0) rectangle (9+6,6);

\draw (9+0,5) to (9+6,5);
\draw (9+1,6) to (9+1,0);

\node[] at (9+0.5,5.5) {\LARGE $+$};
\node[] at (9+1.5,4.5) {\LARGE $0$};
\node[] at (9+3,4.5) {\LARGE $0$};
\node[] at (9+5,4.5) {\LARGE $0$};
\node[] at (9+1.5,3) {\LARGE $0$};
\node[] at (9+1.5,1) {\LARGE $0$};
\node[] at (9+5,1) {\LARGE $\times$};
\node[] at (9+3,1) {\LARGE $\times$};
\node[] at (9+3,5.5) {\LARGE $0$};
\node[] at (9+5,5.5) {\LARGE $0$};
\node[] at (9+5,3) {\LARGE $\times$};
\node[] at (9+3,3) {\LARGE $\times$};
\node[] at (9+0.5,3) {\LARGE $0$};
\node[] at (9+0.5,1) {\LARGE $0$};

\draw[decoration={brace,raise=1.25pt},decorate] (9+0.05,6) -- node[above=6.75pt] {\Large $R_1$} (9+1.95,6);
\draw[decoration={brace,raise=1.25pt},decorate] (9+4.05,6) -- node[above=6.75pt] {\Large $R_t$} (9+5.95,6);
\node[] at (9+2.5,6.5) {\LARGE $.$};
\node[] at (9+3,6.5) {\LARGE $.$};
\node[] at (9+3.5,6.5) {\LARGE $.$};

\end{tikzpicture}
	\end{center}
		\caption{How to diagonalize $X(R_{1:t})$ in Lemma \ref{lemma-X-rotate}}
	\label{figure:convert} 
\end{figure}
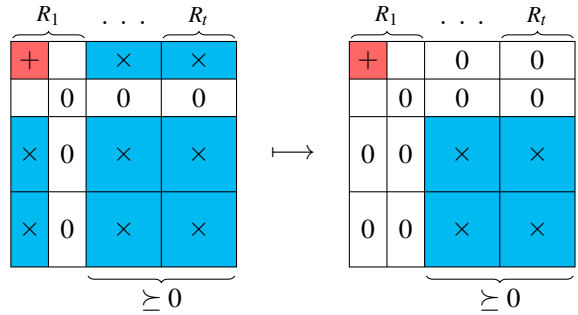
\vspace{-0.7cm}
Next we let $W_1$ be a matrix such that right multiplying $T^{\top} X  T$ by $W_1$ adds 
columns of $T^{\top} X  T$ indexed by $R_1$ to columns indexed by $R_j$ to zero out the 
$T^{\top}XT(R_1, R_j)$ block for all $j > 1.$ Then the $R_{1:t}$ region of $W_{1}^{\top}	T^{\top} X T W_1$ looks like 
in the right picture on Figure \ref{figure:convert}.

We then redefine $T := T W_1 \, $ and $X := T^{\top}XT, \, $ and continue in like fashion with the $R_{2:t}$ diagonal block of $X.$

\end{proof}

The next definition is from the theory of facial reduction \cite{BorWolk:81,Pataki:13}. 
\begin{definition}
	We say that a sequence of symmetric matrices $X_1, \dots, X_t$ is a facial reduction sequence for $\psdn$ if 
	$$
	X_1 \in \psdn, \quad \text{and} \quad X_{i+1} \in \bigl(   \psdn \cap X_1^\perp \cap \dots \cap X_i^\perp \bigr)^* \quad \text{for} ~~ i=1, \dots, t-1.
	$$ 
	Here, for a set $C \subseteq \symn$  we write  $C^* \, = \, \{ Y \, : \, X \bullet Y \geq 0 ~ \text{for  all} ~  X \in C \,   \}$ for its dual cone. 
\end{definition}
Evidently,  if $(X_1, \dots, X_t)$ is in semidefinite echelon form, then 
it is a facial reduction sequence, but the converse is not true in general.

Lemma \ref{claim-V-fr} below follows from Lemma 1 in \cite{liu2017exact}; however, below we give a simpler proof.
\begin{lemma} \label{claim-V-fr} Suppose that $(X_1, \dots, X_t)$ is a facial reduction sequence, and 
	$V$ is an invertible matrix. Then 
	$(V^{\top} X_1 V, \dots, V^{\top} X_{t} V)$ is also a  facial reduction sequence. 
\end{lemma}
\begin{proof} 
Let $(X_1, \dots, X_t)$ be as stated. For brevity, define the map $\V: \symn \rightarrow \symn$ as 
$
\V X  \, = \, V^{\top} X V \, \text{for } \, X \in \symn.
$
Then the conjugate of $\V$ is computed as 
$
\V^* Y = V Y V^{\top}  \, \text{for} \,  \, Y \in \sym{n}. 
$  

Let us fix $i \in \{1, \dots, t-1 \}$ and let 
$Y \in  \psd{n} \cap  ( \V X_1)^\perp     \cap \dots \cap  ( \V X_i)^\perp.$ 
We will show 
\begin{equation} \label{eqn-fr-toprove} 
 \V X_{i+1} \bullet Y \, \geq \, 0,
\end{equation} 
and this will prove our claim. 
From the definition of the conjugate we deduce 
\begin{equation} \label{V*YXi} 
\V^* Y \in \psd{n} \cap X_1^\perp \cap \dots \cap X_i^\perp.
\end{equation}
Hence  
$
 \V X_{i+1} \bullet Y  \, = \,   X_{i+1} \bullet \V^* Y  \, \geq \, 0,
$
where the inequality follows from \eqref{V*YXi} and 
from $(X_1, \dots, X_t)$ being a facial reduction sequence.
Hence \eqref{eqn-fr-toprove} follows and the proof is complete.

\end{proof}

We can now prove the difficult direction in Theorem \ref{thm-weak-ref}.

\begin{proof}[Proof\;(of ``only if"  in Theorem \ref{thm-weak-ref})]

Suppose that \eqref{p} is weakly infeasible 
and 
 Lemma \ref{thm-inf} produced the reformulation  \eqref{equation-p-infeas} 
 with operator 
 $\A^\prime$ and right hand side $b^\prime.$ 
  We claim that  $k \geq 1,$ so to obtain a contradiction, suppose $k=0.$ 
  Then 
  the alternative system of \eqref{equation-p-infeas}  (namely the system 
\eqref{p-alt} with $(\A', b')$ in place of $(\A, b)$) is feasible, in particular, $y=(1,0, \dots, 0)^\top$ is feasible in it. 
Hence \eqref{equation-p-infeas} is strongly infeasible. Thus \eqref{p} is also strongly infeasible, which is the desired contradiction.

Also suppose that  Lemma   \ref{thm-notstrong} produced the reformulation 
 \eqref{equation-p-notstrong}  with 
operator $\A^{\prime\prime}$ and right hand side $b^{\prime\prime};$ 
and it produced  the sequence   $(X_1, \dots, X_{\ell+1})$ which is in semidefinite echelon form and certifies that  \eqref{equation-p-notstrong} is not strongly infeasible.
We claim that  $\ell \geq 1. \, $ Indeed, if $\ell$ were $0, \, $ then 
$X_1$ would be feasible in \eqref{equation-p-notstrong}, hence \eqref{p} would also be feasible, which would be a contradiction.

  As usual, we represent the operator $\A'$ with matrices $A_i'$ and the operator $\A''$ with matrices $A_i''$ for $i=1, \dots, m.$  
Further, following the proof of Lemma \ref{thm-inf}
we assume without loss of generality that the positive definite blocks 
in the $A_i'$ are identities. 

If   \eqref{equation-p-notstrong} is the same as \eqref{equation-p-infeas}, then there is nothing to do.
Otherwise, since both  are reformulations of \eqref{p}, 
we can transform \eqref{equation-p-notstrong}  into \eqref{equation-p-infeas}  if we 
\begin{enumerate} 
\item  \label{ero} perform a sequence of elementary row operations on  the equations $A_i'' \bullet X = b_i''; \, $ 
then 
\item   \label{congr} replace all $A^{\prime\prime}_i$ by $V^{\top} A^{\prime\prime}_i V$  for some invertible matrix $V.$
\end{enumerate} 
Suppose we perform only the elementary row operations, and for simplicity we still call the resulting 
reformulation   \eqref{equation-p-notstrong} with operator $\A^{\prime\prime}$ (represented by matrices $A_i''$) 
and right hand side $b^{\prime\prime}.$ Of course, now $b^\prime = b^{\prime\prime}.$ 
Afterwards equations \eqref{equation-p-notstrong-2}   still hold. At this point we have
\begin{equation} \label{so-complicated} 
\begin{array}{rcll}
A_i' & = & V^{\top} A_i'' V                                                                                                                      & \text{for} ~~ i=1, \dots, m, \\ 
 A^{\prime\prime}_i \bullet X_s  & = &    V^{\top} A^{\prime\prime}_i V \bullet V^{-1} X_s V^{-\top}  & \text{\rm for} ~~ i=1, \dots, m;  \text{ \rm for } ~ s=1, \dots, \ell+1,
\end{array} 
\end{equation}
where the second set of equations follows from the properties of the $\bullet$ product. 
We next perform the following operations:
\begin{equation} \label{eqn-change-Xj}   
X_s  :=  V^{-1} X_s V^{-\top}\quad\text{\rm for} ~~ s=1, \dots, \ell+1.
\end{equation}
Let us consider  the following {\em invariant conditions}, where $j \in \{0, \dots, \ell + 1 \}$: 
\begin{enumerate}[start=1,label={(\bfseries INV-\arabic*):},ref=\bfseries INV-\arabic*, leftmargin=1.6cm]
	\item \label{cond-reformulation}
The semidefinite system \eqref{equation-p-infeas} is 
a 	reformulation of \eqref{p} with properties given in Lemma \ref{thm-inf}.
		 In particular, 
	$(A_1^\prime,\dots, A_{k+1}^\prime)$ is in semidefinite echelon form 
	and 
	$(b_1^\prime, \dots, b_k^\prime, b_{k+1}^\prime) = (0, \dots, 0, -1).$ 
		\item \label{cond-eqn}  
		\begin{equation} \label{eqn-Aiprime-Xj-2}
		\begin{array}{ccl}
		\mathcal{A}^{\prime}X_s &=& 0   ~~  \text{for} ~~  s=1,\dots,\ell \\
		\mathcal{A}^{\prime}X_{\ell+1} &=& b^{\prime}.
		\end{array}
		\end{equation}
		\item \label{cond-fr} $(X_1, \dots, X_{\ell+1})$ is a facial reduction sequence.
		\item \label{cond-sreg} $(X_1, \dots, X_{j})$ is in semidefinite echelon form.
		\end{enumerate}
We claim that all these conditions hold when  $j=0.$ 
Indeed, \eqref{cond-reformulation} holds since \eqref{equation-p-infeas} was constructed in Lemma \ref{thm-inf}. 
Condition \eqref{cond-sreg} holds vacuously.  Condition  \eqref{cond-eqn}  holds by \eqref{so-complicated}, 
by $b^\prime = b^{\prime \prime}, $  and since we performed the operations in \eqref{eqn-change-Xj}. 
Finally, condition  \eqref{cond-fr} holds by  Lemma \ref{claim-V-fr}, since 
we started with $(X_1, \dots, X_{\ell+1})$ being in semidefinite echelon form,
then we performed operations \eqref{eqn-change-Xj}.

The goal is to have the invariant conditions satisfied 
with $j = \ell+1.$ Once that is done, the proof is complete, since 
we can set  \eqref{equation-p-weak} equal to \eqref{equation-p-infeas}. 
So let us assume that $j \in \{0, \dots, \ell \}$ is an integer, and all the invariant conditions hold with $j.$ 

We will perform Step $j$ below, which diagonalizes  a certain block of  $X_{j+1}$ 
and in the process also transforms $\A'$ and the other $X_j.$ We  will prove that afterwards the invariant conditions hold with $j+1.$ Recall the notation \eqref{eqn-RiRj}. 

\begin{enumerate}
	\item[{\bf Step $j$ \, }]
	We assume that $(A_1^\prime, \dots, A_{k+1}^\prime)$ has structure 
		$\{P_1, \dots, P_{k+1}\}$ and set $P_{k+2} := N \setminus P_{1:(k+1)}$. We also assume that 
		$(X_1, \dots, X_j)$ has structure $\{Q_1, \dots, Q_j\}.$ 

By condition  \eqref{cond-sreg} and using an argument like  in the proof of the ``if" direction in Theorem 
\ref{thm-weak-ref} we see that 
\mbox{$\psd{n} \cap X_1^\perp \cap \dots \cap X_j^\perp$} is the set of psd matrices whose 
rows and columns corresponding to $Q_{1:j}$ are zero.
By condition \eref{cond-fr} we have 
$$ 
X_{j+1} \in \bigl( \psd{n} \cap X_1^\perp \cap \dots \cap X_j^\perp \bigr)^*, 
$$ 
hence 
\begin{equation} \label{eqn-Xj+1psd}  
X_{j+1}(N \setminus  Q_{1:j}) \succeq 0. 
\end{equation}
Let us define $R_i \, = \, P_i \setminus Q_{1:j}$ for  $i=1, \dots, k+2.$ 
Then  we  rewrite \eqref{eqn-Xj+1psd} as 
 \begin{equation} \label{eqn-Xj+1psd-Ri} 
 X_{j+1}(R_1 \cup \dots \cup R_{k+2}) \succeq 0. 
 \end{equation} 
 So $X_{j+1}$ looks like on Figure 
 \ref{fig:X_(j+1)-before}, where the red submatrix is positive semidefinite, and the blue and white submatrices are arbitrary.
 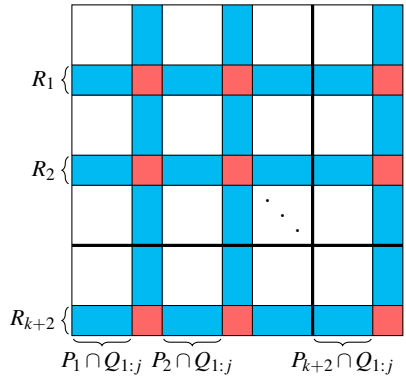
\begin{figure}[H]
 	\begin{center} 
 		\begin{tikzpicture}[scale=0.4, every node/.style={scale=0.6}]
    
    \fill[cyan!80] (0,8) rectangle (11,9);
    \fill[cyan!80] (0,5) rectangle (11,6);
    \fill[cyan!80] (0,0) rectangle (11,1);
    \fill[cyan!80] (2,0) rectangle (3,11);
    \fill[cyan!80] (5,0) rectangle (6,11);
    \fill[cyan!80] (10,0) rectangle (11,11);
    \fill[red!60] (2,0) rectangle (3,1);
    \fill[red!60] (5,0) rectangle (6,1);
    \fill[red!60] (10,0) rectangle (11,1);
    \fill[red!60] (2,5) rectangle (3,6);
    \fill[red!60] (5,5) rectangle (6,6);
    \fill[red!60] (10,5) rectangle (11,6);
    \fill[red!60] (2,8) rectangle (3,9);
    \fill[red!60] (5,8) rectangle (6,9);
    \fill[red!60] (10,8) rectangle (11,9);
    
    \draw[] (2,0) to (2,11);
    \draw[] (3,0) to (3,11);
    \draw[] (5,0) to (5,11);
    \draw[] (6,0) to (6,11);
    \draw[very thick] (8,0) to (8,11);
    \draw[] (10,0) to (10,11);
    \draw[] (0,9) to (11,9);
    \draw[] (0,8) to (11,8);
    \draw[] (0,6) to (11,6);
    \draw[] (0,5) to (11,5);
    \draw[very thick] (0,3) to (11,3);
    \draw[] (0,1) to (11,1);
    \draw (0,0) rectangle (11,11);
    
    \node[] at (6.5,4.5) {\LARGE $.$};
    \node[] at (7,4) {\LARGE $.$};
    \node[] at (7.5,3.5) {\LARGE $.$};
    
    \draw[decoration={brace,raise=1.25pt,mirror},decorate] (0+0.05,0) -- node[below=6.75pt] {\Large $P_1\cap Q_{1:j}$} (1+0.95,0);
    
    \draw[decoration={brace,raise=1.25pt},decorate] (0,8.05) -- node[left=6.75pt] {\Large $R_1$} (0,8.95);
    
    \draw[decoration={brace,raise=1.25pt,mirror},decorate] (3+0.05,0) -- node[below=6.75pt] {\Large $P_2\cap Q_{1:j}$} (4+0.95,0);
    
    \draw[decoration={brace,raise=1.25pt},decorate] (0,5.05) -- node[left=6.75pt] {\Large $R_2$} (0,5.95);
    
    \draw[decoration={brace,raise=1.25pt,mirror},decorate] (8+0.05,0) -- node[below=6.75pt] {\Large $P_{k+2}\cap Q_{1:j}$} (9+0.95,0);
    
    \draw[decoration={brace,raise=1.25pt},decorate] (0,0.05) -- node[left=6.75pt] {\Large $R_{k+2}$} (0,0.95);

    \end{tikzpicture}
 	\end{center}
 	\caption{$X_{j+1}$ before it is transformed}
 	\label{fig:X_(j+1)-before}  
 \end{figure}
We now apply Lemma \ref{lemma-X-rotate} with $X := X_{j+1}, \, t := k+2$  and the $R_i$ just defined.
The goal is to diagonalize $X_{j+1}(R_{1:t}).$  
Let   $T$ be the transformation matrix produced by Lemma \ref{lemma-X-rotate}, then  
\begin{equation} \label{eqn-T-decomp} 
T = I_{R_1, U_1} W_{1}  \dots I_{R_{k+2}, U_{k+2}} W_{k+2}, 
\end{equation}
where 
\begin{enumerate}
	\item the $U_i$ are orthonormal matrices for all $i.$ 
	\item \label{propertyWi} right multiplying an $n \times n$ matrix, say $M, \, $ 
	by $W_i$  
	adds multiples of columns in $M(N, R_i)$ to columns of $M(N, R_j)$ where 
	$j \in \{i+1, \dots, k+2 \}.$     
\end{enumerate}
We perform the operations 
\begin{equation} \label{eqn-T-Ai-Xj}   
\begin{array}{rcll} 
X_s  & := &  T^{\top}  X_s T    &  \text{\rm for} ~~ s = 1, \dots, \ell+1,  \\
A_i^\prime   & := &  T^{-1} A_i^\prime T^{-\top}   &  \text{\rm for} ~~ i=1, \dots, m,   
\end{array} 
\end{equation}
and set 
$$
Q_{j+1} := \{ t ~|~ \text{\rm the }(t,t)\text{ \rm element of }X_{j+1}(N \setminus Q_{1:j}) \text{ \rm is positive} \}.   
$$
\end{enumerate}
We claim that  the invariant  conditions now hold for $j+1.$ Indeed,  \eref{cond-eqn}  holds by how we redefined the $X_s$ and $A_i^\prime$ in \eqref{eqn-T-Ai-Xj}. Condition \eqref{cond-fr}  holds by Lemma \ref{claim-V-fr}. 

We next look at \eref{cond-sreg}. We first show on Figure \ref{figure-Xj+1} how $X_{j+1}$ looks before and after step \eqref{eqn-T-Ai-Xj}. To better see what happened, we permuted the rows and columns of
$X_{j+1}, \, $ so that rows (and columns) indexed by $Q_{1:j}$ come first. 
The $\oplus$ block stands for a positive semidefinite block. 

Step \eqref{eqn-T-Ai-Xj} transforms 
 $X_{j+1}(N \setminus Q_{1:j})$  to be nonnegative diagonal, and in this process $X_{j+1}(Q_{1:j}, N \setminus Q_{1:j})$ (and the symmetric counterpart) is also transformed. We show the changed block in red in Figure \ref{figure-Xj+1}.
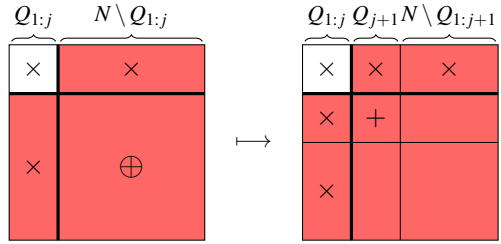
\begin{figure}[H]    
	\begin{center} 
    \begin{tikzpicture}[scale=0.65, every node/.style={scale=0.6}]
	
	\fill[red!60] (0,0) rectangle (4,3);
	\fill[red!60] (1,3) rectangle (4,4);
	
	\fill[red!60] (6,0) rectangle (10,3);
	\fill[red!60] (7,3) rectangle (10,4);
	
	\draw[very thick] (0,3) to (4,3);
	\draw[very thick] (1,0) to (1,4);
	\draw (0,0) rectangle (4,4);
	
	\node[] at (0.5,1.5) {\LARGE $\times$};
	\node[] at (2.5,3.5) {\LARGE $\times$};
	\node[] at (0.5,3.5) {\LARGE $\times$};
	\node[] at (2.5,1.5) {\huge $\oplus$};
	
	\draw[very thick] (6,3) to (10,3);
	\draw[very thick] (7,0) to (7,4);
	\draw (6,2) to (10,2);
	\draw (8,0) to (8,4);
	\draw (6,0) rectangle (10,4);
	
	\node[] at (6.5,1) {\LARGE $\times$};
	\node[] at (9,3.5) {\LARGE $\times$};
	\node[] at (6.5,3.5) {\LARGE $\times$};
	\node[] at (7.5,3.5) {\LARGE $\times$};
	\node[] at (6.5,2.5) {\LARGE $\times$};
	\node[] at (7.5,2.5) {\LARGE $+$};

	\node[] at (5,2) {\Large $\longmapsto$};
	
	\draw[decoration={brace,raise=1.25pt},decorate]
	(0.05,4) -- node[above=6.75pt] {\Large $Q_{1:j}$} (0.95,4);
	
	\draw[decoration={brace,raise=1.25pt},decorate]
	(1.05,4) -- node[above=6.75pt] {\Large $N\setminus Q_{1:j}$} (3.95,4);
	
	\draw[decoration={brace,raise=1.25pt},decorate]
	(6.05,4) -- node[above=6.75pt] {\Large $Q_{1:j}$} (6.95,4);
	\draw[decoration={brace,raise=1.25pt},decorate]
	(7.05,4) -- node[above=6.75pt] {\Large $Q_{j+1}$} (7.95,4);
	\draw[decoration={brace,raise=1.25pt},decorate]
	(8.05,4) -- node[above=6.75pt] {\Large $N\setminus Q_{1:j+1}$} (9.95,4);
	
	\end{tikzpicture}
    \end{center}
	\caption{How step \eqref{eqn-T-Ai-Xj}   changes $X_{j+1}$}
	\label{figure-Xj+1}  
\end{figure}
Next we look at how $X_1, \dots, X_j$ change due to step  \eqref{eqn-T-Ai-Xj}, so we fix $s \in \{1, \dots, j\}.$  Given the factorization of $T$ in \eqref{eqn-T-decomp}, replacing $X_s$ by $T^{\top} X_s T$ amounts to running   Algorithm \ref{algo:transf-Xj+1} below:
\begin{algorithm}[H] 
	\caption{Transforming $X_{s}$}
	\label{algo:transf-Xj+1}
	\begin{algorithmic}[*]
		\For{${t=1:(k+2)}$}  
		\hypertarget{algo3:s}{\State (*) $X_{s} :=   I_{R_t, U_t}^{\top} X_{s} I_{R_t, U_t};$ }
			\hypertarget{algo3:ss}{\State (**) $X_{s} :=   W_t^{\top} X_{s} W_t; $ }
					\EndFor 
	\end{algorithmic}
\end{algorithm}
We claim that after Algorithm \ref{algo:transf-Xj+1} is run, the matrix 
$X_s$ remains in the same shape it was in before.
Suppose this is true after we performed steps \hyperlink{algo3:s}{(*)} and \hyperlink{algo3:ss}{(**)} for $t=1, \dots, q-1,$ where $q \geq 1.$ 

We next perform step \hyperlink{algo3:s}{(*)} with $t=q. \, $ 
This amounts to first multiplying $X_{s}(N, R_q)$ from the right by $U_q, \, $ then 
multiplying $X_{s}(R_q, N)$ from the left by $U_q^{\top}. \, $  
Since $R_1 \cup \dots \cup R_{k+2} = N \setminus Q_{1:j}, \, $ we see that 
$$R_q \subseteq  N \setminus Q_{1:j} \subseteq N \setminus Q_{1:s}.$$
We depict $X_s$ on Figure \ref{fig:transf-X-s}, with the affected parts shaded in red and conclude that $X_s$ remains in the shape it was in before step \eqref{eqn-T-Ai-Xj}.  Note that on Figure \ref{fig:transf-X-s} we permuted the rows and columns of $X_s$ 
so that $X_s(Q_{1:s})$ is in the upper left corner. 
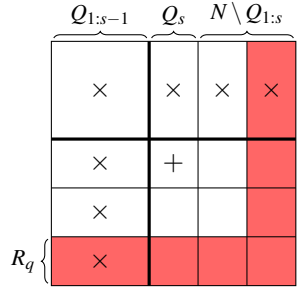
\begin{figure}[H]
	\begin{center} 
    \begin{tikzpicture}[scale=0.65, every node/.style={scale=0.6}]
	
	\fill[red!60] (4,0) rectangle (5,5);
	\fill[red!60] (0,0) rectangle (5,1);
	
	\draw[very thick] (0,3) to (5,3);
	\draw[very thick] (2,0) to (2,5);
	\draw (0,0) rectangle (5,5);
	
	\draw (0,2) to (5,2);
	\draw (0,1) to (5,1);
	\draw (3,0) to (3,5);
	\draw (4,0) to (4,5);
	
	\node[] at (2.5,2.5) {\LARGE $+$};
	
	\node[] at (1,4) {\LARGE $\times$};
	\node[] at (2.5,4) {\LARGE $\times$};
	\node[] at (3.5,4) {\LARGE $\times$};
	\node[] at (4.5,4) {\LARGE $\times$};
	\node[] at (1,2.5) {\LARGE $\times$};
	\node[] at (1,1.5) {\LARGE $\times$};
	\node[] at (1,0.5) {\LARGE $\times$};
	
	\draw[decoration={brace,raise=1.25pt},decorate]
	(0.05,5) -- node[above=6.75pt] {\Large $Q_{1:s-1}$} (1.95,5);
	
	\draw[decoration={brace,raise=1.25pt},decorate]
	(2.05,5) -- node[above=6.75pt] {\Large $Q_s$} (2.95,5);
	
	\draw[decoration={brace,raise=1.25pt},decorate]
	(3.05,5) -- node[above=6.75pt] {\Large $N\setminus Q_{1:s}$} (4.95,5);
	
	\draw[decoration={brace,raise=1.25pt},decorate] (0,0.05) -- node[left=6.75pt] {\Large $R_q$} (0,0.95);
	
	\end{tikzpicture}
    \end{center}
	\caption{How steps \protect\hyperlink{algo3:s}{(*)} and \protect\hyperlink{algo3:ss}{(**)} in Algorithm \ref{algo:transf-Xj+1}  change $X_{s},$ where $s \leq j.$}
	\label{fig:transf-X-s}  
\end{figure}
We next perform step \hyperlink{algo3:ss}{(**)} with $t=q. \, $ Multiplying $X_{s}$ from the right by 
$W_q$ adds multiples of columns in $X_s(N, R_q)$ 
to columns   in $X_s(N, R_{q'})$  where $q' \in \{q+1, \dots, k+2 \}.$ 
Then we perform the analogous row operations. 
Thus Figure \ref{fig:transf-X-s}   again tells us that 
$X_{s}$ remains in the same shape. 

We thus conclude that condition (\ref{cond-sreg}) holds after step \eqref{eqn-T-Ai-Xj} with $j+1$ instead of $j.$ 

We next prove that condition  \eqref{cond-reformulation}   remains unchanged after step  \eqref{eqn-T-Ai-Xj} is executed, so we 
look at how the $A_i^\prime$ change. 
Let us fix $i \in \{1, \dots, m \}.$ 

Given the decomposition \eqref{eqn-T-decomp}, we have 
\begin{equation} \label{eqn-T-T} 
\begin{array}{rcl} 
T^{-\top}  & = & I_{R_1 U_1}^{-\top} W_{1}^{-\top}  \dots I_{R_{k+2}, U_{k+2}}^{-\top} W_{k+2}^{-\top}, \\
T^{-1}  & = & W_{k+2}^{-1}  I_{R_{k+2} U_{k+2}}^{-1}   \dots W_{1}^{-1}  I_{R_{1}, U_{1}}^{-1}.  
\end{array} 
\end{equation} 
We also know that for all $t \in \{1, \dots, k+2  \}$ by the the definition of  $ \, I_{R, G} $ and 
by  $U_t^{\top} = U_t^{-1}$ the following  hold: 
\begin{equation}
\begin{array}{rcl} 
I_{R_t, U_t}^{-1} & = & I_{R_t, U_t^{\top}} \\ 
I_{R_t, U_t}^{-\top} & = & I_{R_t, U_t}. 
\end{array}
\end{equation}  
Thus, given  the decomposition \eqref{eqn-T-T},  performing step  \eqref{eqn-T-Ai-Xj} on $A_i^\prime$ amounts to running Algorithm \ref{algo:transf-Ai} below.
\begin{algorithm}[H]
	\caption{Transforming  $A_i^\prime$}
	\label{algo:transf-Ai}
	\begin{algorithmic}
		\For{${t=1:(k+2)}$}  
		\hypertarget{algo4:s}{\State(*) $A_{i}^\prime :=   I_{R_t, U_t^{\top}} A_i^\prime I_{R_t, U_t};$ }
		\hypertarget{algo4:ss}{\State (**) $A_i^\prime :=   W_t^{-1} A_i^\prime W_t^{-\top}; $ }
		\EndFor 
	\end{algorithmic}
\end{algorithm}
We claim that after Algorithm \ref{algo:transf-Ai} is run, the matrix $A_i'$ remains in the shape it was in before.

Suppose this is true after we performed steps \hyperlink{algo4:s}{(*)} and \hyperlink{algo4:ss}{(**)} for $t=1, \dots, q-1,$ where $q \geq 1.$  
Next we perform \hyperlink{algo4:s}{(*)} with $t=q.$ 
We need to keep in mind that $q \in \{1, \dots, k+2 \}$ and $R_q \subseteq P_q.$  

We distinguish three cases.
\begin{itemize}[leftmargin=1cm]
	\item[$\mathbf{q=i}$]   
		We show on Figure \ref{fig:transf-Ai-q=i}  the $A_i^\prime$ matrix before and after step \hyperlink{algo4:s}{(*)}. The changed portion is in red. 
	
	First the submatrix $A_i^\prime(N, R_q)$ is multiplied from the right by 	$U_q, \, $ then the 
	submatrix $A_i^\prime(R_q, N)$ is multiplied from the left by 	$U_q^{\top}. \, $
	   
	Thus $A_i^\prime(R_q)=I$ is replaced by $U_q^{\top}  I U_q = I, \,$ so it remains an identity. 
	In summary, $A_i^\prime$ has the same form before and after step \hyperlink{algo4:s}{(*)}. 
		\begin{figure}[H]
		\begin{center} 
        \begin{tikzpicture}[scale=0.65, every node/.style={scale=0.6}]
		
		\fill[red!60] (2,0) rectangle (3,6);
		\fill[red!60] (0,3) rectangle (6,4);
		
		\draw[very thick] (0,2) to (6,2);
		\draw[very thick] (0,4) to (6,4);
		\draw[very thick] (2,0) to (2,6);
		\draw[very thick] (4,0) to (4,6);
		\draw (0,0) rectangle (6,6);
		
		\draw (0,3) to (6,3);
		\draw (3,0) to (3,6);
		
		\node[] at (2.5,3.5) {\LARGE $I$};
		\node[] at (3.5,2.5) {\LARGE $I$};
		
		\node[] at (1,3.5) {\LARGE $\times$};
		\node[] at (2.5,5) {\LARGE $\times$};
		\node[] at (1,2.5) {\LARGE $\times$};
		\node[] at (3.5,5) {\LARGE $\times$};
		
		\node[] at (5,5) {\LARGE $\times$};
		\node[] at (1,1) {\LARGE $\times$};
		\node[] at (1,5) {\LARGE $\times$};
		
		\draw[decoration={brace,raise=1.25pt},decorate]
		(0.05,6) -- node[above=6.75pt] {\Large $P_{1:i-1}$} (1.95,6);
		
		\draw[decoration={brace,raise=1.25pt},decorate]
		(2.05,6) -- node[above=6.75pt] {\Large $P_i$} (3.95,6);
		
		\draw[decoration={brace,raise=1.25pt},decorate]
		(4.05,6) -- node[above=6.75pt] {\Large $N\setminus P_{1:i}$} (5.95,6);
		
		\draw[decoration={brace,raise=1.25pt},decorate] (0,3.05) -- node[left=6.75pt] {\Large $R_q$} (0,3.95);
		
		\end{tikzpicture}
        \end{center}
		\caption{How step \protect\hyperlink{algo4:s}{(*)} in Algorithm \ref{algo:transf-Ai} changes $A_i^\prime$, when $q=i$}
		\label{fig:transf-Ai-q=i} 
	\end{figure}

 \item[$\mathbf{q<i}$]   We show on  Figure \ref{fig:transf-Ai-q<i} the $A_i^\prime$ matrix, before and after step \hyperlink{algo4:s}{(*)}. The changed portion is in red.

	First the submatrix $A_i^\prime(N, R_q)$ is multiplied from the right by 	$U_q, \, $ then the 
	submatrix $A_i^\prime(R_q, N)$ is multiplied from the left by 	$U_q^{\top}. \, $
	
Again, $A_i^\prime$ has the same form before and after step \hyperlink{algo4:s}{(*)}. 
	
	\begin{figure}[H]
		\begin{center} 
        \begin{tikzpicture}[scale=0.65, every node/.style={scale=0.6}]
		
		\fill[red!60] (1,0) rectangle (2,6);
		\fill[red!60] (0,4) rectangle (6,5);
		
		\draw[very thick] (0,2) to (6,2);
		\draw[very thick] (0,4) to (6,4);
		\draw[very thick] (2,0) to (2,6);
		\draw[very thick] (4,0) to (4,6);
		\draw (0,0) rectangle (6,6);
		
		\draw (0,5) to (6,5);
		\draw (1,0) to (1,6);
		
		\node[] at (3,3) {\LARGE $I$};
		
		\node[] at (0.5,4.5) {\LARGE $\times$};
		\node[] at (1.5,5.5) {\LARGE $\times$};
		\node[] at (1.5,3) {\LARGE $\times$};
		\node[] at (1.5,1) {\LARGE $\times$};
		\node[] at (0.5,3) {\LARGE $\times$};
		\node[] at (3,5.5) {\LARGE $\times$};
		\node[] at (3,4.5) {\LARGE $\times$};
		\node[] at (5,4.5) {\LARGE $\times$};
		
		\node[] at (5,5.5) {\LARGE $\times$};
		\node[] at (0.5,1) {\LARGE $\times$};
		\node[] at (0.5,5.5) {\LARGE $\times$};
		\node[] at (1.5,4.5) {\LARGE $\times$};
		
		\draw[decoration={brace,raise=1.25pt},decorate]
		(0,4.05) -- node[left=6.75pt] {\Large $R_q$} (0,4.95);
		
		\draw[decoration={brace,raise=1.25pt},decorate]
		(0.05,6) -- node[above=6.75pt] {\Large $P_{1:i-1}$} (1.95,6);
		
		\draw[decoration={brace,raise=1.25pt},decorate]
		(2.05,6) -- node[above=6.75pt] {\Large $P_i$} (3.95,6);
		
		\draw[decoration={brace,raise=1.25pt},decorate]
		(4.05,6) -- node[above=6.75pt] {\Large $N\setminus P_{1:i}$} (5.95,6);
		
		\end{tikzpicture}
        \end{center}
		\caption{How step \protect\hyperlink{algo4:s}{(*)} in Algorithm \ref{algo:transf-Ai} changes $A_i^\prime$, when $q<i$}
		\label{fig:transf-Ai-q<i} 
	\end{figure}

 	\item[$\mathbf{q>i}$] We show on Figure \ref{fig:transf-Ai-q>i} the $A_i^\prime$ matrix, before and after step \hyperlink{algo4:s}{(*)}. The changed portion is again in red.
 	
 	First the submatrix $A_i^\prime(N, R_q)$ is multiplied from the right by 	$U_q, \, $ then the 
 	submatrix $A_i^\prime(R_q, N)$ is multiplied from the left by 	$U_q^{\top}. \, $
 	
 Yet again, $A_i^\prime$ has the same form before and after step \hyperlink{algo4:s}{(*)}. 
	
	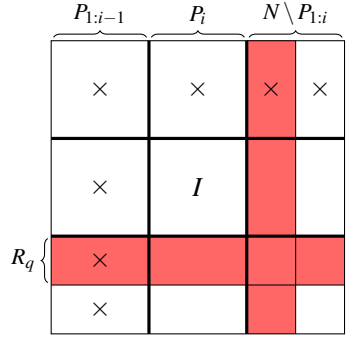
\begin{figure}[H]
 		\begin{center} 
        \begin{tikzpicture}[scale=0.65, every node/.style={scale=0.6}]
	 			
	 			\fill[red!60] (4,0) rectangle (5,6);
	 			\fill[red!60] (0,1) rectangle (6,2);
	 			
	 			\draw[very thick] (0,2) to (6,2);
	 			\draw[very thick] (0,4) to (6,4);
	 			\draw[very thick] (2,0) to (2,6);
	 			\draw[very thick] (4,0) to (4,6);
	 			\draw (0,0) rectangle (6,6);
	 			
	 			\draw (0,1) to (6,1);
	 			\draw (5,0) to (5,6);
	 			
	 			\node[] at (3,3) {\LARGE $I$};
	 			
	 			\node[] at (1,5) {\LARGE $\times$};
	 			\node[] at (3,5) {\LARGE $\times$};
	 			\node[] at (4.5,5) {\LARGE $\times$};
	 			\node[] at (5.5,5) {\LARGE $\times$};
	 			\node[] at (1,3) {\LARGE $\times$};
	 			\node[] at (1,1.5) {\LARGE $\times$};
	 			\node[] at (1,0.5) {\LARGE $\times$};
	 			
	 			\draw[decoration={brace,raise=1.25pt},decorate]
		        (0,1.05) -- node[left=6.75pt] {\Large $R_q$} (0,1.95);
	 			
	 			\draw[decoration={brace,raise=1.25pt},decorate]
	 			(0.05,6) -- node[above=6.75pt] {\Large $P_{1:i-1}$} (1.95,6);
	 			
	 			\draw[decoration={brace,raise=1.25pt},decorate]
	 			(2.05,6) -- node[above=6.75pt] {\Large $P_i$} (3.95,6);
	 			
	 			\draw[decoration={brace,raise=1.25pt},decorate]
	 			(4.05,6) -- node[above=6.75pt] {\Large $N\setminus P_{1:i}$} (5.95,6);
	 			
	 			\end{tikzpicture}
        \end{center}
 		\caption{How step \eqref{eqn-T-Ai-Xj} changes $A_i^\prime$, when $q>i$}
 		\label{fig:transf-Ai-q>i} 
 	\end{figure}
	 	
\end{itemize}

Next we perform step \hyperlink{algo4:ss}{(**)} with $t=q.$  We recall that right multiplying $A_i'$ by $W_q$ adds columns of 
$A_i'(N, R_q)$ to columns in $A_i'(N, R_{q'})$ where $q < q'.$ By the algebraic description 
of $W_q$  (after the statement of Lemma \ref{lemma-X-rotate})  we see that 
multiplying $A_i^\prime$ from the right by $W_q^{-\top}$  adds 
columns of $A_i'(N, R_{q'})$ to columns in $A_i'(N, R_{q})$ where $q < q'.$ (Multiplying $A_i'$ from the left by 
$W_q^{-1}$ works analogously on the rows of $A_i'.$) 
  
Recall that $R_{q} \subseteq P_{q}$ and $R_{q'} \subseteq P_{q'}.$ Thus  Figures 
\ref{fig:transf-Ai-q=i},  	\ref{fig:transf-Ai-q<i}, and \ref{fig:transf-Ai-q>i} tell us that 
$A_i^\prime$ remains in the same shape as it was in before step \hyperlink{algo4:ss}{(**)}. 
Thus condition \eqref{cond-reformulation} holds, and the proof is complete.

\end{proof}

\section{Our problem library and computational tests}
\label{section-computational}

Using the results of the previous sections, we now show how to generate a library of weakly infeasible SDPs.
We  accompany our SDPs with an intuitive visualization and 
examine whether their infeasibility can be 
recognized by the prominent SDP solvers MOSEK   \cite{andersen2000mosek}   and SDPA-GMP 
\cite{fujisawa2008sdpaGMP}. 

Libraries of weakly infeasible SDPs are available \cite{liu2017exact,Waki:12} the latter of these 
was generated using the Lasserre relaxation of polynomial optimization problems.
On the other hand, any weakly infeasible SDP is a possible output of our Algorithm \ref{weaksdpalgo}, so our current library includes instances that are unlikely to appear in any previous  collection. 

\paragraph{The instances} 

We constructed all $80$ instances 
using Algorithm \ref{weaksdpalgo} and 
split them into two classes: clean and messy.

\begin{itemize}
	\item We constructed our clean instances by steps \ref{alg-1}--\ref{alg-3} 
	of Algorithm \ref{weaksdpalgo}, without using the reformulation step of step \ref{alg-4}, 
	so our clean 
	 SDPs are in the echelon form \eqref{equation-p-weak}.    
	
\item From each clean instance  we created a corresponding messy instance 
as follows. 
 We applied 
 elementary row operations that 
 we represent by an $m \times m$ integral matrix $F = (f_{ij}),$ then 
 a congruence transformation  that we represent by an 
$n \times n$ integral matrix $T.$ 
	
	That is, if \eqref{p} is a clean instance, then 
 in the corresponding messy instance 	constraint $i$ is
	$$
	T^{\top} (\sum_{j=1}^m f_{ij} A_j)  T \bullet X \, = \, \sum_{j=1}^m f_{ij} b_j. 
	$$
	
\end{itemize} 

We further categorize the instances as ``miniature", ``small", ``medium",  and ``large", with parameters given in Table \ref{table-param}.  
 In each category the clean and messy instances are in one-to-one correspondence. For example, from the ``mini, clean, 9" instance we constructed the ``mini, messy, 9" instance.  
 In all instances all data is integral,  so their weak infeasibility can be verified by  hand, 
 in exact arithmetic.

\renewcommand{\arraystretch}{1.2} 
\begin{table}
	\centering
	\begin{tabular}{|c||c|c|c|c|}
		\hline
		& Miniature & Small & Medium & Large \\
		\hline\hline
		$k$ & 1 & 3 & 3 & 4 \\
		\hline
		$\ell$ & 1 & 1--3 & 2--4 & 2--4 \\
		\hline
		$n$ & 3 & 5--15 & 25--40 & 120--240 \\
		\hline
		$m$ & 2--5 & 4--7 & 4--7 & 5--8 \\
		\hline
	\end{tabular}
	\caption{Parameters of our weakly infeasible SDPs: $k+1$ being the length of the infeasibility certificate, $\ell+1$ the length of the not-strong infeasibility certificate, $n$ the matrix order, and $m$ the number of constraints.} 
	\label{table-param} 
\end{table}

\paragraph{Data storage}

For convenience we give our instances in three formats. 

\begin{enumerate}
	\item In the ``.mat'' files (in Matlab format)  
	 the  $A_1, \dots, A_m$ are stored as rows of a  matrix $A$ and   
	  the  $X_1, \dots, X_{\ell+1}$ are  stored as rows of a  matrix $X.$ 
	  These files also contain the right hand side $b.$ 
	
	For each messy instance the files also include the matrices $F$ and $T$ that were used 
	to create it  from the corresponding clean instance.
	
	\item The ``.cbf'' and ``.dat-s'' files contain the same SDPs. 
	The   ``.cbf'' files can be directly read by  MOSEK and 
	the ``.dat-s'' files can be directly read by SDPA-GMP.

\end{enumerate}

The ``.jpg'' files in the ``image" subdirectories 
	contain visualizations of the matrices $A_i,X_j$ and $F$ and $T$ for each problem. Matrices 
	$A_1, \dots, A_{k+1}$ and $X_1, \dots, X_{\ell+1}$ are color coded, just like in Figure 
	\ref{figure-large-weak}. 

\paragraph{Computational testing}
	
	For computational testing we selected the SDP solvers MOSEK and SDPA-GMP as representative
	 industry standards. MOSEK is currently the only commercially available  SDP solver, it is fast and accurate on most industrial problems, however,  it has limited numerical precision. On  the other hand, SDPA-GMP can carry out calculations with precision $10^{-200}.$ 
	
	The results are in Table \ref{resultstable}, where we reported a solver's output as ``correct" if it marked 
	an instance  as infeasible.  
\begin{table}[H]
\centering
\begin{tabular}{|c|c|c|c|c|c|c|}
    \hline
     & \multicolumn{2}{c|}{Miniature} & \multicolumn{2}{c|}{Small} & \multicolumn{2}{c|}{Medium/Large} \\
     \cline{2-7}
     & Clean & Messy & Clean & Messy & Clean & Messy \\
    \hline
    MOSEK & 0 & 0 & 0 & 0 & 0 & 0 \\
    \hline
    SDPA-GMP & 10 & 10 & 0 & 2 & 0 & 0 \\
    \hline\hline
    Total correct & 10 & 10 & 10 & 10 & 20 & 20 \\
    \hline
\end{tabular}
\caption{\label{resultstable}Number of infeasible instances correctly identified by MOSEK and SDPA-GMP}
\end{table}

We see that while  MOSEK failed to identify infeasibility of any of the SDP instances, SDPA-GMP correctly 
identified the infeasibility of all miniature and of some small instances.  However, both solvers failed on the 
the larger instances.

	Besides standalone SDP solvers we also tested  two implementations of facial reduction. 
	The first is the  Sieve-SDP algorithm \cite{zhu2019sieve}, which works very well 
	if the constraint matrices of an SDP are in semidefinite echelon form. Not surprisingly, Sieve-SDP quickly proved infeasibility 
of all clean instances, but failed on all messy instances. The second implementation is the facial reduction method of Permenter and Parrilo 
\cite{permenter2014partial}  which performed very similarly. 

The problem instances  are available from the webpage of the first author. 

\section{Discussion and conclusion} 
\label{section-discussion} 

We presented an echelon form of weakly infeasible SDPs that permits us to construct any weakly infeasible SDP and any bad projection of the psd cone by a combinatorial algorithm. We conclude with a discussion.

First we recall  normal forms  of  other types of SDPs and linear maps; these normal forms are similar in spirit, but much easier to derive. For example,  in  \cite{Pataki:17} and in Section 2.2.1  in \cite{pataki2019characterizing} we produced a normal form of 
{\em well-behaved} semidefinite systems of the form 
\begin{equation} \label{eqn-wellbehaved}   
\sum_{i=1}^m x_i A_i \preceq B,
\end{equation}
 and called the normal form a {\em good reformulation.}
The system \eqref{eqn-wellbehaved} with $B = 0$ is well behaved iff 
$\A \psdn$ is closed, hence we obtain a normal form of   linear maps that 
carry $\psdn$ to a {\em closed} set. 
This normal form also  permits us to construct any such linear map. 
 In contrast, to derive such a normal (echelon) form 
of weakly infeasible SDPs and of  maps that carry $\psdn$ to a {\em nonclosed} set, 
we needed a much more technical proof. 

Next, we reinterpret Theorem \ref{thm-weak-ref} in two equivalent forms.

The first  interpretation is a  ``sandwich theorem" which we state in terms of weak infeasibility of $H \cap \psdn$ where the affine subspace $H$ is defined in \eqref{eqn-H}. 
\begin{theorem}
	The semidefinite program  $H \cap \psdn$ is weakly infeasible if and only if there are positive integers $k$ and $\ell, \, $ an invertible matrix $T, \,$ and sequences 
	$(A_1', \dots, A_{k+1}')$ and $(X_1, \dots, X_{\ell+1}),$ both in semidefinite echelon form 
	such that   
					\begin{equation*}
			\begin{array}{rcl}
			X_{\ell+1} + \lin \{ \, X_1, \dots, X_\ell \, \} & \subseteq & T^{\top} H T \\
			& \subseteq & \{ X ~:~ \, A_1' \bullet X = \dots = A_k' \bullet X = 0,~  A_{k+1}' \bullet X = -1   \}.
			\end{array}
			\end{equation*}
	\qed 
\end{theorem}
Here $\lin \{ \, X_1, \dots, X_\ell \, \}$ stands for the linear span of $X_1, \dots, X_\ell, \,$ and 
$T^\top HT$ for the set $\{\,  T^{\top} X T~:~ X \in H  \, \}.$ 

The second  interpretation is a ``factorization" result. 
Suppose that \eqref{equation-p-weak} in Theorem \ref{thm-weak-1} was obtained  by elementary row operations that we represent  by an $m \times m$ matrix $G$ and by 
congruence transformations that we represent by an $n \times n$ matrix $T,$ see the discussion in Remark  \ref{remark-HT}.
We define the map $\T: \symn \rightarrow \symn$ as $\T(X) = TXT^{\top}$ for $X \in \symn.$ 
Then  in Theorem \ref{thm-weak-ref}  
the operators $\A$ and  $\A'$ and vectors $b$ and  $b'$ are related as 
\begin{equation*} 
\A' = G \A \T, \, b' = Gb, 
\end{equation*}
so $(\A,b)$ can be ``factorized"  into $(\A', b')$ (and vice versa).

	Next we comment on computing the echelon form \eqref{equation-p-weak}. 
		In general, we do not have an efficient or stable method to do that, since we would need to solve the SDPs that arise in  Lemma \ref{lemma-infeas} in exact arithmetic.
		However, 
		the complexity of even finding a feasible solution to an SDP is unknown (see  Remark \ref{remark-HT}).
			 Thus the echelon form \eqref{equation-p-weak}  is  similar in spirit to the Jordan normal  form of a matrix, which also must be computed in exact arithmetic \cite{golub1976ill}. At the same time, even though they are nontrivial to compute, 
			 both our echelon form, and the Jordan normal form yield both theoretical and practical insights.

We finally mention some intriguing research questions. First, it may be of interest to interpret our results from the viewpont of projective geometry, in the spirit of Naldi and Sinn's recent paper \cite{naldi2021conic}. Second, recall again Example \ref{example-motzkin}: here  the SDP from minimizing the sum-of-squares  (SOS) relaxation of the Motzkin polynomial is weakly infeasible, and is in the echelon form of \eqref{equation-p-weak} 
without having to reformulate it.  It would be interesting to see whether the same is true of other sum-of-squares SDPs.

\noindent{\bf Acknowledgments} We gratefully acknowledge the support of the 
National Science Foundation, award DMS-1817272. We thank Bernd Sturmfels and Didier Henrion for stimulating discussions. We also thank the anonymous referees for their helpful and detailed  suggestions.

%
%

\bibliographystyle{plain}      
\bibliography{weaksdp-FOCM}   


\end{document}